\documentclass[reqno]{amsart}

\usepackage[
left=3cm,% left margin
right=3cm,% right margin
top=3cm, % top margin
bottom=3cm,% bottom margin
a4paper]{geometry}
\usepackage{setspace}
\usepackage{stmaryrd} %%%% for llbracket and rrbracket

\let\tempone\itemize
\let\temptwo\enditemize
\renewenvironment{itemize}{\tempone \vspace{5pt}\addtolength{\itemsep}{0.5\baselineskip}}{\vspace{5pt} \temptwo}

\let\tempenum\enumerate
\let\tempenumtwo\endenumerate
\renewenvironment{enumerate}{\tempenum \vspace{5pt} \addtolength{\itemsep}{0.5\baselineskip}}{ \vspace{5pt} \tempenumtwo}

\makeatletter
\let\origsection\section
\renewcommand\section{\@ifstar{\starsection}{\nostarsection}}
\newcommand\nostarsection[1]
{\sectionprelude\origsection{#1}\sectionpostlude}
\newcommand\starsection[1]
{\sectionprelude\origsection*{#1}\sectionpostlude}
\newcommand\sectionprelude{%
	\vspace{1em} %%% before
}
\newcommand\sectionpostlude{%
	\vspace{1em}   %%%% after
}
\makeatother

\newcommand\Item[1][]{%
	\ifx\relax#1\relax  \item \else \item[#1] \fi
	\abovedisplayskip=0pt\abovedisplayshortskip=0pt~\vspace*{-\baselineskip}}

\makeatletter
\let\origsubsection\subsection
\renewcommand\subsection{\@ifstar{\starsubsection}{\nostarsubsection}}
\newcommand\nostarsubsection[1]
{\sectionprelude\origsubsection{#1}\subsectionpostlude}
\newcommand\starsubsection[1]
{\subsectionprelude\origsubsection*{#1}\subsectionpostlude}
\newcommand\subsectionprelude{%
	\vspace{0.25em} %%% before
}
\newcommand\subsectionpostlude{%
	\vspace{0.0em}   %%%% after
}
\makeatother

\usepackage[bb=boondox]{mathalfa}

\usepackage[T1]{fontenc}
\usepackage{lmodern}
\usepackage[utf8]{inputenc}
\usepackage{bm}
\usepackage{mathrsfs}
\usepackage{bbm}
\usepackage{faktor} 
\usepackage{mathtools}

\usepackage{esvect}
\usepackage[UKenglish]{babel}

\usepackage{amssymb}
\usepackage{amsthm}
\usepackage{amsmath}
\usepackage{mathrsfs}
\usepackage{tikz-cd}

\usepackage{tikz}

\usetikzlibrary{decorations.markings}
\usetikzlibrary{shapes.geometric, patterns, angles, quotes}

\usetikzlibrary{hobby}

\usepackage{amsmath,amssymb,amsthm,graphicx,amscd,amsfonts,latexsym}
\usepackage{pstricks}
\usepackage[normalem]{ulem}

\usepackage{lscape}
\usepackage{hyperref}

\usepackage{tikz, cleveref}
\usetikzlibrary{arrows}
\usetikzlibrary{decorations}
\usetikzlibrary{shapes}
\usetikzlibrary{decorations.pathreplacing,decorations.markings} %arrows on smooth curves in tikz pictures

\usetikzlibrary{calc,intersections,through,backgrounds}
\usepackage{graphicx}
\usepackage{caption} %for pictures in morphism section
\usepackage{float}
\usepackage{upgreek}

\usetikzlibrary{shapes.geometric,positioning}
\usetikzlibrary{arrows,decorations.pathmorphing,decorations.pathreplacing}

\usepackage{ifthen} % needed for tikz

\makeatletter
\newcases{rightalignedcases}{\quad}{%
	\hfil$\m@th\displaystyle{##}$\hfil}{$\m@th\displaystyle{##}$\hfil}{\lbrace}{.}
\makeatother

\newcounter{pos} % needed for tikz
\tikzset{									% for pictures
	initcounter/.code={\setcounter{pos}{0}},
	style between/.style n args={3}{
		postaction={
			initcounter,
			decorate,
			decoration={
				show path construction,
				curveto code={
					\addtocounter{pos}{1}
					\pgfmathtruncatemacro{\min}{#1 - 1}
					\ifthenelse{\thepos < #2 \AND \thepos > \min}{
						\draw[#3]
						(\tikzinputsegmentfirst)
						..
						controls (\tikzinputsegmentsupporta) and (\tikzinputsegmentsupportb)
						..
						(\tikzinputsegmentlast);
					}{}
				}
			}
		},
	},
}

\tikzset{
	clip even odd rule/.code={\pgfseteorule}, % Credit to Andrew Stacey 
	invclip/.style={
		clip,insert path=
		[clip even odd rule]{
			[reset cm](-\maxdimen,-\maxdimen)rectangle(\maxdimen,\maxdimen)
		}
	}
}

\usepackage{epstopdf}

\makeatletter

\DeclareMathOperator{\Hom}{Hom}

\DeclareMathOperator{\Perf}{Perf}

\newcommand{\Dfd}[1]{\mathsf{D}_{\operatorname{fd}}(#1)}

\newcommand{\1}{\mathbf{1}}
\newcommand{\DPic}{\operatorname{\mathcal{D}Pic}}

\tikzset{
	set arrow inside/.code={\pgfqkeys{/tikz/arrow inside}{#1}},
	set arrow inside={end/.initial=>, opt/.initial=},
	/pgf/decoration/Mark/.style={
		mark/.expanded=at position #1 with
		{
			\noexpand\arrow[\pgfkeysvalueof{/tikz/arrow inside/opt}]{\pgfkeysvalueof{/tikz/arrow inside/end}}
		}
	},
	arrow inside/.style 2 args={
		set arrow inside={#1},
		postaction={
			decorate,decoration={
				markings,Mark/.list={#2}
			}
		}
	},
}

\makeatletter
\tikzset{commutative diagrams/.cd,arrow style=tikz,diagrams={>=latex'}}\tikzset{join/.code=\tikzset{after node path={%
			\ifx\tikzchainprevious\pgfutil@empty\else(\tikzchainprevious)%
			edge[every join]#1(\tikzchaincurrent)\fi}}}
\makeatother

\tikzset{>=stealth',every on chain/.append style={join},
	every join/.style={->}}

\tikzset{every loop/.style={min distance=25mm,in=50,out=100,looseness=5}}

\tikzcdset{
	cells={font=\everymath\expandafter{\the\everymath\displaystyle}},
}

\newtheorem{prf}{Proof}[section]

\theoremstyle{remark}
\newtheoremstyle{ownTheoremStyle}% name of the style to be used
{1em}% measure of space to leave above the theorem. E.g.: 3pt
{1em}% measure of space to leave below the theorem. E.g.: 3pt
{\itshape}% name of font to use in the body of the theorem
{}% measure of space to indent
{\bfseries}% name of head font
{.}% punctuation between head and body
{ }% space after theorem head; " " = normal interword space
{}% Manually specify head

\newtheoremstyle{ownDefinitionStyle}% name of the style to be used
{1em}% measure of space to leave above the theorem. E.g.: 3pt
{1em}% measure of space to leave below the theorem. E.g.: 3pt
{}% name of font to use in the body of the theorem
{}% measure of space to indent
{\bfseries}% name of head font
{.}% punctuation between head and body
{ }% space after theorem head; " " = normal interword space
{}% Manually specify head  

\theoremstyle{ownTheoremStyle}

\newtheorem{thm}[prf]{Theorem}

\newtheorem{Introthm}{Theorem}

\newtheorem{lem}[prf]{Lemma}
\newtheorem{prp}[prf]{Proposition}

\newtheorem{cor}[prf]{Corollary}

\theoremstyle{ownDefinitionStyle}

\newtheorem*{convention}{Convention}

\newtheorem{definition}[prf]{Definition}
\newtheorem{rem}[prf]{Remark}

\numberwithin{equation}{section}

\newcommand{\cA}{\mathcal{A}}
\newcommand{\cB}{\mathcal{B}}
\newcommand{\cC}{\mathcal{C}}
\newcommand{\cD}{\mathcal{D}}

\newcommand{\cH}{\mathcal{H}}
\newcommand{\cK}{\mathcal{K}}

\newcommand{\cM}{\mathcal{M}}
\newcommand{\cN}{\mathcal{N}}

\newcommand{\cP}{\mathcal{P}}
\newcommand{\cR}{\mathcal{R}}
\newcommand{\cQ}{\mathcal{Q}}

\newcommand{\cU}{\mathcal{U}}

\newcommand{\cW}{\mathcal{W}}

\newcommand{\op}{\operatorname{op}}

\newcommand{\HHH}{\operatorname{HH}}
\newcommand{\HH}{\operatorname{H}}

\newcommand{\rot}{\operatorname{rot}}

\newcommand{\para}[1]{{#1}_{\odot}}

\newcommand{\Fuk}{\operatorname{Fuk}}

\title[Hochschild cohomology of graded gentle algebras and intrinsic formality]{Hochschild cohomology of graded gentle algebras and intrinsic formality}
\author{Sebastian Opper}
\address{Department of Algebra,
	Charles University,
	Sokolovská 49/83, Praha 8, Czechia}
\email{opper@karlin.mff.cuni.cz }
\begin{document}%\parindent 0pt
	\maketitle

	\begin{abstract}
	We describe the (bigraded) Hochschild cohomology of graded gentle algebras along with the Gerstenhaber bracket and cup product. In particular, this yields a description of the Hochschild cohomology of partially wrapped Fukaya categories of surfaces in the sense of Haiden-Katzarkov-Kontsevich which have at least one stop. Our results are an important ingredient in the author's recent description of the derived Picard group of partially wrapped Fukaya categories and graded gentle algebras. As another application we provide a characterisation of intrinsically formal graded gentle algebras under mild assumptions.
	\end{abstract}
	\setcounter{tocdepth}{1}
	\tableofcontents
	
	\section*{Introduction}
	\noindent Our interest in Hochschild cohomology stems from a general connection between the Hochschild cohomology of a dg category $\cC$ or $A_\infty$-category and its \textit{derived Picard group} $\DPic(\cC)$. This group can be thought of as an `enhanced' variant of the autoequivalence group of the derived category of $\cC$ which takes into account the `higher' structure that a dg or $A_\infty$-structure provides. The author showed recently in \cite{OpperIntegration} that, similar to the relationship between a Lie group and its Lie algebra, there is an exponential map $$\exp\colon \HHH^1_+(\cC,\cC) \rightarrow \DPic(\cC),$$ which is defined on a subspace of the first Hochschild cohomology. The exponential shares many properties with its classical counterpart and satisfies an analogue of the Baker-Campbell-Hausdorff formula. The Gerstenhaber bracket therefore allows one to express the group multiplication in the image of $\exp$ explicitly and an understanding of the Hochschild cohomology and Gerstenhaber bracket of an algebra or category provides insight into the structure of the corresponding derived Picard group. Naturally, the Hochschild cohomology of graded gentle algebras and the exponential play a vital role in the author's recent description \cite{OpperGradedGentle} of the derived Picard groups of graded gentle algebras which uses the results of this article.
	
	 First introduced in the 1980s \cite{AssemSkowronski} the class of gentle algebras received increased attention thanks to their connection to \textit{partially wrapped Fukaya categories of surfaces} as introduced in \cite{BocklandtMirrorSymmetryPuncturedSurfaces, HaidenKatzarkovKontsevich}. A subclass of graded gentle algebras then appear as the cohomology algebras of formal generators in these categories and as a result, our description of the Hochschild cohomology for graded gentle algebras also provides a description of the Hochschild cohomology of such Fukaya categories in all cases except those in \cite{BocklandtMirrorSymmetryPuncturedSurfaces}. In this sense, our work is complementary to the results in \cite{BocklandtVanDeKreeke} and \cite{ChaparroSchrollSolotarSuarezAlvarez}. The former describes the Hochschild cohomology of Fukaya categories of punctured surfaces considered in \cite{BocklandtMirrorSymmetryPuncturedSurfaces} and while the latter computes the Hochschild cohomology (and the entire Tamarkin-Tsygan calculus) of ungraded gentle algebras. The Hochschild cohomology of graded gentle algebras also played an important role in the recent work \cite{BarmeierSchrollWang} which computes their Hochschild cohomology in degree $2$.
	 
	To define a Fukaya category in the sense of \cite{HaidenKatzarkovKontsevich}, one starts with a \textit{graded marked surface}, an oriented surface $\Sigma$ together with a distinguished subset $\cM$ of its boundary and a line field. A line field allows one to assign a unique integer to every boundary component of $\Sigma$, its so-called \textit{winding number}. To a graded marked surface and certain  auxiliary data, the authors of \cite{HaidenKatzarkovKontsevich} then assign their partially wrapped Fukaya category. The category turns out to be independent from the auxiliary data which corresponds to the choice of a generator of the Fukaya category and which has an underlying graded gentle algebra. In the other direction, starting from a graded gentle algebra $A$ one can construct a graded marked surface $\Sigma_A$ \cite{HaidenKatzarkovKontsevich, LekiliPolishchukGentle, OpperPlamondonSchroll, PaluPilaudPlamondon}.
	
	To simplify stating our main results, we will exclude two exceptional cases from our discussion in the introduction, namely gentle algebras those whose underlying quiver is either a loop or the Kronecker quiver. Our first result describes the entire Hochschild cohomology of a graded gentle algebra.	
	\begin{Introthm}\label{Introthm: Hochschild cohomology}
	Let $A$ be a graded gentle algebras, possibly neither homologically smooth nor proper and let $\Sigma_A$ denote its graded marked surface. Then $\HHH^{\bullet}(A,A)$ has a $\Bbbk$-linear (Schauder) basis whose elements are in bijection with the set consisting of the following elements:
	\begin{itemize}
		\item the elements $f_1, \dots, f_l$ of a basis of $\HH^1(\Sigma_A, \Bbbk)$;
		\item a pair of elements $\cN^0(\gamma)$ and $\cN^1(\gamma)$ for every non-contractible orientation-preserving path $\gamma\colon S^1 \rightarrow \partial \Sigma_A$ (up to homotopy) with winding number $\omega$ such that $\omega \cdot \operatorname{char} \Bbbk$ is even and such that the boundary component $B$ which is the image of $\gamma$ is either fully marked or unmarked;
		\item an element $[B]$ for every boundary component $B$ of $\Sigma_A$ with exactly one stop. 
	\end{itemize}
\noindent Moreover, $f_i \in \HHH^1(A,A)$, $\cN^s(\gamma) \in \HHH^{\omega+s}(A,A)$ and $[B] \in \HHH^{\omega}(A,A)$, where $\omega$ denotes the winding number of $B$. Here, a boundary component $B$ is \textit{fully marked} if it lies in the distinguished subset $\cM$ of $\partial \Sigma_A$ and \textit{unmarked} if it is disjoint from it. A \textit{stop} on $B$ is by definition a connected component of $B \setminus \cM$.
\end{Introthm}
\noindent In fact, we prove a slightly stronger result and determine the bigraded Hochschild cohomology of every graded gentle algebra, cf.~\Cref{thm: appendix Hochschild basis}. In what follows, we identify the elements in \Cref{Introthm: Hochschild cohomology} with their corresponding basis elements of $\HHH^{\bullet}(A,A)$. Through the relationship \cite{HaidenKatzarkovKontsevich} between partially wrapped Fukaya categories and homologically smooth graded gentle algebras, \Cref{Introthm: Hochschild cohomology} therefore provides a description of the Hochschild cohomology of all partially wrapped Fukaya categories associated with a graded marked surface with at least one stop. Our next result describes the Gerstenhaber bracket of $\HHH^{\bullet}(A,A)$ in terms of the basis. For the description of the cup product see \Cref{thm: cup product}.
\begin{Introthm}\label{IntrothmB}
Let $A$ be as before. The following holds: 
\begin{enumerate}
	\item For $1 \leq i \leq l$ and $v \in \HHH^{\bullet}(A,A)$ of the form $\cN^s(\gamma)$ or $[B]$, $$[f_i, v]= f_i(B) v,$$ where $f_i$ is evaluated at $B$ through the identification of  $B \cong S^1$ with an element of $\HH_1(\Sigma_A, \Bbbk)$ and the isomorphism $\HH^1(\Sigma_A, \Bbbk)\cong \Hom_{\Bbbk}(\HH_1(\Sigma_A, \Bbbk), \Bbbk)$.
	\item For every primitive non-contractible closed path $\gamma\colon S^1 \rightarrow \Sigma_A$ and all $m, n \geq 1$ such that the paths $\gamma^m$ and $\gamma^n$ satisfy the conditions of \Cref{Introthm: Hochschild cohomology}, 
	\begin{displaymath}
	\begin{aligned}
	[\cN^1(\gamma^m), \cN^1(\gamma^n)] & =(m-n)\cdot  \cN^1(\gamma^{m+n}) \\
	[\cN^0(\gamma^m), \cN^1(\gamma^n)] & = m  \cdot \cN^0(\gamma^{m+n}).
	\end{aligned}
	\end{displaymath}
	\item All other brackets between basis elements vanish.
\end{enumerate}
\end{Introthm} 
\noindent Finally, the bigraded variant of \Cref{Introthm: Hochschild cohomology} allows us to give a characterisation of intrinsically formal graded gentle algebras. We recall that a graded $\Bbbk$-algebra $A$ is \textit{intrinsically formal} if any dg or $A_\infty$-algebra $B$ with cohomology $A$ is already formal, that is, quasi-isomorphic to $A$. In other words, there are no non-trivial higher structures (=$A_\infty$-algebra structures) on intrinsically formal algebras. This property is often useful, for example when comparing generators of different triangulated categories. If then the graded endomorphism rings of the two generators are isomorphic and intrinsically formal, then the corresponding triangulated categories are automatically equivalent.

 \begin{Introthm}[{\Cref{thm: intrinsic formality graded gentle algebras}}]
Let $A$ be a graded gentle algebra, possibly neither homologically smooth nor proper and suppose $\Sigma_A$ contains no boundary components which has a single stop and winding number $2$. Then $A$ is intrinsically formal if and only if $\Sigma_A$ has no unmarked boundary component with winding number $2$.
\end{Introthm}
\noindent  We strongly expect that this is also true when dropping the additional assumption on $\Sigma_A$, cf.~\Cref{rem: intrinsic formality}.\medskip

\noindent While working on this paper, we were recently informed about the work \cite{BianSchrollSolotarWangWen} by Xiuli Bian, Sibylle Schroll, Andrea Solotar, Xiao-Chuang Wang and Can Wen who describe the Hochschild cohomology of graded skew-gentle algebras along with its Gerstenhaber bracket and cup product via different methods.

\subsection*{Acknowledgments}This project was supported by the Primus grant PRIMUS/23/SCI/006.

\subsection*{Notation}
Given a quiver $Q$, we will denote by $Q_0$ and $Q_1$ its sets of vertices and arrows respectively and denote by $\Bbbk Q$ the associated path algebra. We further denote by $s,t \colon Q_1 \rightarrow Q_0$  the functions which assign to an arrow $\alpha \in Q_1$ its start and end vertex. Arrows will be composed from right to left, that is, if $s(\beta)=t(\alpha)$ we denote by $\beta \alpha$ their concatenation and hence their product in $\Bbbk Q$. The trivial path of length zero associated to a vertex $x \in Q_0$ will be denoted by $e_x$.

\section{Gentle algebras and Fukaya categories}
\subsection{Graded gentle algebras and partially wrapped Fukaya categories}
\noindent
\begin{definition}A \textbf{graded gentle quiver} is a pair $(Q,I)$ consisting of a finite quiver $Q$ whose arrows are graded by integers and ideal $I \subseteq \Bbbk Q$ which satisfy all of the following conditions:
	\begin{enumerate}
		\item For every vertex $x \in Q_0$, there are at most two arrows $\alpha, \beta \in Q_1$ and at most two arrows $\gamma, \delta \in Q_1$ such that $s(\alpha)=x=s(\beta)$ and $t(\gamma)=x=t(\delta)$.
		\item For every arrow $\alpha \in Q_1$, there exists at most one arrow $\beta$  such that $t(\alpha)=s(\beta)$ and $\beta \alpha \in I$ and at most one arrow $\gamma \in Q_1$ such that $s(\alpha)=t(\gamma)$ and $\gamma \alpha \in I$.
		\item  For every arrow $\alpha \in Q_1$, there exists at most one arrow $\beta$ such that $t(\alpha)=s(\beta)$ and $\beta \alpha \not\in I$ and at most one arrow $\gamma \in Q_1$ such that $s(\alpha)=t(\gamma)$ and $\gamma \alpha \not\in I$.
		\item $I$ is generated by a quadratic monomials $\alpha \beta$, that is, paths of length $2$.
	\end{enumerate}
A \textbf{graded gentle algebra} is a graded $\Bbbk$-algebra which is isomorphic as a graded algebra to a $\Bbbk Q / I$, where $(Q,I)$ is a graded gentle quiver and $\Bbbk Q / I$ inherits its grading from the given grading on the arrows. For any path $q$ in $Q$, we denote by $|q|$ its degree.
\end{definition}
\begin{convention}
Unless explicitly stated otherwise, we will always exclude gentle quivers from our discussion which are a loop with a single vertex or the Kronecker quiver. These cases behave somewhat special and will be treated separately.	
\end{convention} 
\noindent Graded gentle algebras appear as formal generators of certain Fukaya categories.
\begin{definition}[{\cite{HaidenKatzarkovKontsevich}}]A \textbf{graded marked surface} is a triple $(\Sigma, \cM, \eta)$ consisting of
	\begin{enumerate}
		\item a smooth oriented surface $\Sigma$ with boundary,
		\item  a non-empty compact subset $\cM \subseteq \partial \Sigma$ such that $\cM$ is homeomorphic to a disjoint union of copies of the unit circle $S^1$ and closed intervals $[a, b]$, $a < b$.
		\item a line field $\eta$ on $\Sigma$, that is, a section of the projectivised tangent bundle.
	\end{enumerate} 
\noindent A \textbf{stop} is a connected component of $\partial \Sigma \setminus \cM$. A boundary component $B$ is \textbf{fully marked} if $B \subseteq \cM$ and \textbf{unmarked} if $B \cap \cM = \emptyset$.
\end{definition}
\noindent By definition, a line field $\eta$ smoothly attaches to every point $x \in \Sigma$ a $1$-dimensional subspace $\eta(x) \subseteq T_x\Sigma$ of its tangent space at $x$. In that sense, one may think of them as an orientation-free generalisation of nowhere vanishing vector field. Indeed, every nowhere vanishing vector field $\chi$ on $\Sigma$ induces a line field which pointwise assigns the $1$-dimensional subspace generated by $\chi(x) \in T_x\Sigma$. However, not every line field arises in this way. Every line field $\eta$ gives rise to a \textbf{winding number function} $\omega=\omega_{\eta}$ which assigns to every immersed closed curve $\gamma\colon S^1 \rightarrow \Sigma$ an integer $\omega(\gamma) \in \mathbb{Z}$ called the \textbf{winding number}. Formally, this is the intersection number of the submanifold $\eta(\Sigma)$ and the derivative of $\gamma$ both viewed as submanifolds of the projectivised tangent bundle. In particular, when $\gamma$ is replaced by the inverse loop running in opposite directions, its winding number switches signs. More precisely, the $1$-dimensional subspace generated by the derivative of $\gamma$ at any point of its domain $S^1$ an embedding of $S^1$ into the projectivised tangent bundle. In this note, we will exclusively talk about winding numbers of boundary components of a graded marked surface. By definition, the winding number $\omega_{\eta}(B)$ of a boundary component $B \subseteq \partial \Sigma$, is the winding number of any orientation preserving embedding $S^1 \hookrightarrow \partial \Sigma$ associated to $B \cong S^1$. The invariance properties of $\omega_{\eta}$ guarantee that $\omega_{\eta}(B)$ is a well-defined integer.\medskip

\noindent To any graded marked surface $(\Sigma, \cM, \eta)$ for which $\cM$ intersects each boundary component, \cite{HaidenKatzarkovKontsevich} defined a pretriangulated $A_\infty$-category, the \textbf{partially wrapped Fukaya category}. The set $\cM$ and the line field $\eta$ are dropped from  the notation and the Fukaya category is denoted by $\Fuk(\Sigma)$. The concrete definition of $\Fuk(\Sigma)$ depends on the auxiliary choice of an \textit{arc system}, a collection of embedded paths in the surface  with end points in $\cM$. The higher multiplications in $\Fuk(\Sigma)$, are then defined in terms of so-called \textit{disk sequences}, certain immersed disks whose boundary is made up of the arcs of the arc system.  The justification to speak of \textit{the} Fukaya category of $\Sigma$ comes from the fact that $\Fuk(\Sigma)$ is independent up to equivalence from this auxiliary data. We refer the reader to \cite{HaidenKatzarkovKontsevich} for further details.
\begin{thm}[{\cite{HaidenKatzarkovKontsevich}}]\label{thm: gentle algebras as generators}Let $(\Sigma, \cM, \eta)$ be a graded marked surface with at least one stop. Then there exist a homologically smooth\footnote{A dg algebra $A$ is \textbf{homologically smooth} if the diagonal bimodule $A$ is a compact object in the derived category of dg $A$-bimodules $\cD(A \otimes A^{\otimes})$.} graded gentle algebra such that
	\begin{displaymath}
	\Fuk(\Sigma) \cong \Perf(A).
	\end{displaymath}
\end{thm}
\noindent Here the vertices of the underlying quiver of $A$ are in bijection with a chosen arc system. One can also go the other way and construct a graded marked surface $\Sigma_A$ from a graded gentle algebra $A$, cf.\cite{PaluPilaudPlamondon, HaidenKatzarkovKontsevich, LekiliPolishchukGentle, OpperPlamondonSchroll} and \Cref{rem: surface models graded gentle algebras}. The boundary components of $\Sigma_A$ and its markings are encoded combinatorially through the Avella-Alaminos-Geiss invariant, cf.\Cref{sec: boundary components AAG}.

\begin{rem}\label{rem: surface models graded gentle algebras}
Starting from a graded gentle algebra $A$ which is homologically smooth, one can construct a canonical graded marked surface $\Sigma_A$ such that $\Fuk(\Sigma_A) \simeq \Perf(A)$ \cite{LekiliPolishchukGentle}. A similar result exists for proper graded gentle algebras: if $B$ is a finite-dimensional graded gentle algebra, there exists a graded marked surface $\Sigma_B$ such that $\Dfd{B}$ models a certain (derived) completion of $\Fuk(\Sigma_B)$ \cite{OpperPlamondonSchroll, BoothGoodbodyOpper}. This is an expression of the Koszul duality between proper and homologically smooth graded gentle algebras \cite{OpperPlamondonSchroll} and generation results for $\Dfd{B}$ proved in \cite{BoothGoodbodyOpper}. Under mild conditions on $B$, or equivalently, $\Sigma_B$, $\Perf(B)$ is equivalent to \textit{infinitesimal Fukaya category} of $\Sigma_B$ which coincides with the subcategory $\Dfd{\Fuk(\Sigma_B)} \subseteq \Fuk(\Sigma_B)$ in these cases, cf.\cite[Section 9]{BoothGoodbodyOpper} for a more detailed discussion.
\end{rem}

\begin{rem}\label{rem: complete derived invariants}The graded gentle algebra describing the Fukaya category in \Cref{thm: gentle algebras as generators} is far from unique. It was shown in \cite{AmiotPlamondonSchroll, OpperDerivedEquivalences} for ungraded gentle algebras, in \cite{JinSchrollWang} for homologically smooth graded gentle algebras, and in \cite{OpperGradedGentle} for graded gentle algebras which are homologically smooth or proper, that the perfect derived categories of such graded gentle algebras are equivalent if and only if $\Sigma_A \cong \Sigma_B$ as graded marked surfaces. The latter means that there is a diffeomorphism $\psi\colon\Sigma_A \rightarrow \Sigma_B$ such that $\psi(\cM_A)=\cM_B$ and such that the line field defined by pulling back $\eta_B$ along $\psi$ is homotopic to $\eta_A$. An equivalent formulation is that derived equivalence classes of graded gentle algebras are in bijection with orbits of $(\Sigma, \cM, \eta)$ under the action of the mapping class group of $(\Sigma, \cM)$ consisting of isotopy classes of diffeomorphisms $\Sigma \rightarrow \Sigma$ which map $\cM$ to itself. By \cite{LekiliPolishchukGentle}, there is a small number of quite explicit invariants which allow to \textit{decide} if the graded marked surfaces of $A$ and $B$ lie in the same orbit and hence allow to decide if $\Perf(A) \simeq \Perf(B)$. Due to the results of \cite{BoothGoodbodyOpper}, this is equivalent to the existence of an equivalence $\Dfd{A} \simeq \Dfd{B}$.
\end{rem}

\subsection{Paths and anti-paths}
\noindent We fix a graded gentle quiver $(Q,I)$ and its associated graded gentle algebra $A=\Bbbk Q/ I$.

\begin{definition}A path $p=\alpha_1 \cdots \alpha_l$, $\alpha_i \in Q_1$, in $Q$ is called 
	\begin{itemize}
		\item an \textbf{anti-path of $A$} if $l(p)=0$ or $\alpha_i \alpha_{i+1} \not\in I$ for all $1 \leq i < l$.
		\item \textbf{path of $A$} if $l(p)=0$ or $\alpha_i \alpha_{i+1} \in I$ for all $1 \leq i < l$.
	\end{itemize}
An anti-path (resp.~path)  $p$ of $A$ as above is \textbf{complete} if $s(p)=t(p)$ and $\alpha_l \alpha_1 \not \in I$ (resp.~$\alpha_l \alpha_1 \in I$). In particular, any positive power of a complete anti-path (resp.~path) is a complete anti-path (resp.~path). We denote by $\rot(p)=\alpha_2 \cdots \alpha_l \alpha_1$ the \textbf{rotation} of $p$ and by $\rot^n(p)$ the $n$-fold iteration of this process which is defined for negative $n$ in the obvious way with $\rot^{-n}$ being inverse to $\rot^n$. The equivalence classes of the equivalence  relation generated by $p \sim \rot(p)$ will be called rotation classes.   An anti-path (resp.~path) $p$ of $A$ is \textbf{maximal} if it is not properly contained in a longer anti-path (resp.~path) of $A$. For a maximal anti-path $p$, there exists at most one path $\para{p}$ of $A$ such that $s(\para{p})=s(p)$, $t(\para{p})=t(p)$ and such that the first (resp.~last) arrow of $\para{p}$ is different from the first (resp.~last) arrow of $p$. If no such path exists, we write $\para{p}=\emptyset$.
\end{definition}
\noindent By definition, any rotation of a complete (anti-)path $p$ is again a complete (anti-)path and we call the smallest number $n \geq 1$ such that $\rot^n(p)=p$ the \textbf{period} of $p$ while we refer to the quotient $e(p) \coloneqq \frac{l(p)}{n}$ as the \textbf{exponent} of $p$. A complete (anti-)path with exponent $1$ will be called \textbf{primitive}.

 For a complete path $q$ of $A$, define $\omega(q)\coloneqq - |q|$, $n(q)\coloneqq \infty$ and for an a complete anti-path $p$ define $\omega(p) \coloneqq |p|- l(p)$ and $n(p)=0$. Note that in both cases $\omega(x^n)=n \omega(x)$ and $\omega(-)$ only depends on the rotation class of a complete path or anti-path.

\subsection{Smoothness and properness}

As before we regard graded algebras as dg algebras with trivial differential. We recall that a dg algebra $B$ is \textbf{homologically smooth} if the diagonal dg $B$-bimodule associated to $B$ is is a compact object in $\cD(B \otimes B)$, the derived category of dg $B$-bimodules. Moreover, $B$ is \textbf{proper} if its total cohomology $\HH^{\bullet}(A)$ is a finite-dimensional vector space. In particular, a graded algebra is proper in this sense if and only if it is finite-dimensional. There easy straightforward characterisations of these properties for graded gentle algebras.

\begin{prp}[{\cite[Proposition 3.4]{HaidenKatzarkovKontsevich}}]Let $(Q,I)$ be a graded gentle quiver and let $A$ denote the associated graded gentle algebra. Then,
	\begin{enumerate}
		\item $A$ is homologically smooth if and only if  $(Q,I)$ contains no complete anti-paths, or equivalently, $\Sigma_A$ has no unmarked boundary components;
		\item $A$ is proper if and only if $(Q,I)$ contains no complete paths, or equivalently, $\Sigma_A$ has no fully marked boundary components. 
	\end{enumerate}
\end{prp}

\subsection{The AAG invariant and boundary components}\label{sec: boundary components AAG}

For the next definition, we formally adjoin inverses $\alpha^{-1}$ of arrows $\alpha \in Q_1$ to our quiver so that $s(\alpha^{-1})=t(\alpha)$ and $t(\alpha^{-1})=s(\alpha)$. This yields an extended quiver $Q^{\pm} \supseteq Q$ in which a path is called \textbf{reduced} if it does not contain a subpath of the form $\alpha \alpha^{-1}$ or $\alpha^{-1} \alpha$ for $\alpha \in Q_1$. The assignment $\alpha \in \alpha^{-1}$ can be extended uniquely to an anti-involution\footnote{An \textit{anti-involution} of a quiver $E$ is an isomorphism $\iota\colon E \rightarrow E^{\op}$ such that $\iota^{\op} \circ \iota=\operatorname{Id}_E$, where $E^{\op}$ denotes the opposite quiver obtained from $E$ by inverting the orientations of all arrows and $\iota^{\op}\colon E^{\op} \rightarrow (E^{\op})$ denotes the opposite of $\iota$. In other words, $\iota$ induces a bijection on vertices and paths from $E$ to itself but maps a path from $x$ to $y$ in $E$ to a path from $\iota(y)$ to $\iota(x)$ in $E$.} $Q^{\pm}$. In particular, a path $p=p_1 \cdots p_l$ in $Q \subseteq Q^{\pm}$ $p^{-1}=p_l^{-1} \cdots p_1^{-1} \in Q^{\pm}$.

\begin{definition}\label{def: boundary cycle} A \textbf{boundary cycle} of $A$ is either a complete anti-path of $A$, a complete path of $A$, or a closed path $c=p_1 q_1 p_2 \cdots p_r q_r$ in $Q^{\pm}$, $r \geq 1$, with $p_{r+1}=p_1$ and $q_{r+1}=q_1$ such that
\begin{itemize}
	\item each $q_i$ is a path of $A$,
	\item each $p_i$ is the inverse of an anti-path of $A$,
	\item $c$ is reduced 
	\item for all $1 \leq i \leq r$, $s(p_i)=t(q_i)$ and $s(q_i)=t(p_{i+1})$.
\end{itemize}
Define $n(c)\coloneqq r$ and $\omega(c)\coloneqq r + \sum_{i=1}^r (|q_i| - l(p_i)))$. One identifies boundary cycles which agree up to rotation, that is, $c=p_1q_1 \cdots p_rq_r$ is considered equivalent to $p_2q_2 \cdots p_{r}q_rp_1q_1$.
\end{definition}
\noindent  Boundary cycles were also called \textit{combinatorial boundary components} in \cite{LekiliPolishchukGentle}. We note that the condition that $c$ is reduced is equivalent to the assumption that  first arrow of $p_i$ is not the inverse of the last arrow of $q_i$ and that the last arrow of $p_{i+1}$ is not the inverse of the first arrow of $q_i$ for each $1 \leq i \leq r$. 
\begin{prp}[{cf.~\cite{LekiliPolishchukGentle, OpperDerivedEquivalences, ChaparroSchrollSolotarSuarezAlvarez}}]
There is a bijection $c \mapsto B_c$ between the set of equivalence classes of boundary cycles of $A$ and the set of boundary components of $\Sigma_A$.
\end{prp}
Under this bijection, complete paths corresponds to boundary components $C$ such that $C \subseteq \cM$ and complete anti-paths correspond to boundary components $C \subseteq \partial \Sigma_A$ such that $C \cap \cM = \emptyset$. For every other boundary cycle $c$ the number $r$ in \Cref{def: boundary cycle} coincides with the number of connected components of $\cM \cap C$, where $C$ denotes the boundary component associated to $c$.  In all cases, the winding number of $C$ agrees with $\omega(c)$, cf.~\cite[Proposition 3.2.1]{LekiliPolishchukGentle} and \cite[Corollary 2.17]{OpperDerivedEquivalences}. Next, we recall the Avella-Alaminos-Geiss invariant, first defined in \cite{AvellaAlaminosGeiss}, and then extended in \cite{LekiliPolishchukGentle}.
\begin{definition}The \textbf{Avella-Alaminos-Geiss invariant} of $A$ (AAG invariant for short) is the function $\phi_A\colon (\mathbb{N}_{\geq 0} \cup \{\infty\}) \times \mathbb{Z} \rightarrow \mathbb{N}_{\geq 0}$ such that $\phi_A(n,m)$ is the cardinality of the set of boundary cycles $c$ of $A$ such that $n(c)=n$ and $m=n-\omega(c)$ whenever $n \neq \infty$.  Moreover, $\phi_A(\infty, m)$ is by definition the number of complete anti-paths $c$ with $\omega(c)=-m$. 
\end{definition}
\noindent If $A$ is homologically smooth or proper, it is known that $\phi_A$ is a derived invariant, that is, if $B$ is a graded gentle algebra such that $\Perf(A) \simeq \Perf(B)$, then $\phi_A=\phi_B$. This follows from the characterisation of derived equivalences between such graded gentle algebras, cf.~\Cref{rem: complete derived invariants}. A special role in this note is played by the values $\phi_A(n,m)$ when $n \in \{0,1, \infty\}$ which admit an interpretation in terms of $\HHH^{\bullet}(A,A)$. We note that the conditions of \Cref{def: boundary cycle} imply that if $c$ is a boundary cycle with $n(c)=1$, then either, $p_1$ is a maximal anti-path and $\para{p_1^{-1}}=q_1$, or $p_1$ is trivial and $q_1$ is a closed maximal path of $A$. 
\begin{rem}
\Cref{def: boundary cycle} differs slightly from the definition of $\phi_A$ in \cite{LekiliPolishchukGentle} which specifies $n(c)=0$ if $c$ is a complete path or a complete anti-path as opposed to $n(p)=\infty$ in \Cref{def: boundary cycle} for an anti-path $p$. Thus, if $A$ is neither homologically smooth nor proper, then \Cref{def: boundary cycle} is a slightly finer quantity from which their version can be recovered. Of course, if $A$ is homologically smooth or proper however, there are no complete paths or no complete anti-paths and \Cref{def: boundary cycle} is essentially equivalent to the one of \cite{LekiliPolishchukGentle} after replacing the values $\infty$ of $\phi_A$ with $0$ 
instead. Under Koszul duality one has the equality $\phi_{A^!}(n,m)=\phi_A(n,m)$ if $n \not \in \{0,\infty\}$ and $\phi_{A^!}(n, m)=\phi_A(\frac{1}{n}, m)$ for $n \in \{0, \infty\}$ where $\frac{1}{0}\coloneqq\infty$ and $\frac{1}{\infty}\coloneqq 0$.
\end{rem}
\begin{rem}
The set of boundary cycles can be computed algorithmically by starting with any vertex and a maximal path or anti-path $p_1$ of $A$ which contains it. Then $p_1$ can be completed uniquely to a boundary cycle. In particular, every vertex appears in at most two boundary cycles. For more details, see \cite{AvellaAlaminosGeiss} or \cite{LekiliPolishchukGentle}.
\end{rem}

\begin{rem}If $A$ is a proper graded gentle algebra, then $\phi_A(m,n)$ counts orbits of isomorphism classes of special collections of $(n,n-m)$-fractional Calabi-Yau objects, that is, objects $X \in \Perf(A)$ such that $\mathbb{S}^nX \cong X[n-m]$, where $\mathbb{S}$ denotes a Serre functor. This interpretation of $\phi_A$ is found originally in \cite{AvellaAlaminosGeiss} in the ungraded case and \cite{OpperDerivedEquivalences} for arbitrary gradings. If $(n,m) \neq (1,1)$, then all $(n,n-m)$-fractional Calabi-Yau objects are special in this sense and the derived Morita invariance of Hochschild cohomology gives another proof of the derived Morita invariance of $\phi_A(n,m)$ for all $(n,m) \neq (1,1)$. 
\end{rem}

\section{Hochschild cohomology of graded gentle algebras}

\subsection{Bigraded Hochschild cohomology of a graded algebra}\label{sec: Bigraded Hochschild cohomology}
\noindent  We follow the sign conventions from \cite{RoitzheimWhitehouse} which specialise to those in \cite{ChaparroSchrollSolotarSuarezAlvarez}. Throughout the section, we fix a graded algebra $B$ which we view as a dg algebra with trivial differential. Its Hochschild complex $C(B)$ is the complex whose $d$-th homogeneous component is given by $$\prod_{n \geq 0}\Hom_{\Bbbk}^{d-n}(B^{\otimes n}, B).$$ This is naturally a normed vector space for which the subsets $\{W_jC(B)\}_{j \geq 0}$ defined via $$W_jC(B)=\prod_{n \geq j}\Hom_{\Bbbk}^{d-n}(B^{\otimes n}, B)$$ form a basis of neighbourhoods of the zero vector, similar to the usual adic topology on a local ring. This turns $C(B)$ into a Banach space. The Hochschild differential is the unique continuous differential on $C(B)$ which is determined on $f\colon B^{\otimes n} \rightarrow B$ by the formula
\begin{equation}\label{eq: differential Bar resolution}
	d(f)=(-1)^{|f|}\Big( m_B(\mathbf{1} \otimes f) + (-1)^{n+1} m_B(f \otimes \mathbf{1}) + \sum_{i=0}^{n-1} (-1)^{i+1} f(\mathbf{1}^{\otimes i} \otimes m_B \otimes \mathbf{1}^{n-1-i})\Big),
\end{equation} 
\noindent where $m_B\colon B^{\otimes 2} \rightarrow B$ denotes the multiplication, $\mathbf{1}$ denotes the identity of $B$, $|f|$ denotes the degree of homogenenity of $f$. Tensor products $g \otimes h$ of homogenenous maps $g$ and $h$ between graded vector spaces are evaluated on homogenenous elements $u, v$. using the Koszul sign rule
\begin{displaymath}
	g \otimes h(u \otimes v) = (-1)^{|u| |g|} g(u) \otimes h(v).
\end{displaymath}
\noindent The Hochschild complex is the product totalisation of a natural double complex $C^{\bullet, \bullet}(B)$. A homogeneous element in $C^{\bullet, \bullet}(B)$ of bidegree $(n, d) \in \mathbb{N} \times \mathbb{Z}$ is a homogeneous linear map $f\colon B^{\otimes n} \rightarrow B$ of degree $n+d$ which we will also refer to as the \textbf{total degree}. of $f$. The $d$-differential of $C^{\bullet,\bullet}(B)$ is trivial while the $n$-differential is simply the Hochschild differential which respects is homogenenous of degree $0$ with respect to the $d$-degree. Consequently, there is a canonical isomorphism
\begin{displaymath}
	\HHH^i(B,B) \cong\prod_{n+t=i}\HHH^{n,t}(B,B),
\end{displaymath} 
\noindent where $\HHH^{n,t}(B,B)$ denotes the $n$-th cohomology of the complex $C^{\bullet, t}(B)$ with respect to the $n$-differential (=Hochschild differential). $\HHH^{\bullet}(B,B)$ inherits the structure of a Banach space from $C(B)$. We will use this fact to speak about infinite sums in $\HHH^{\bullet}(B,B)$, or equivalently, the infinite products of elements $(x_n)_{n \geq 0}$ with $x_n \in \HHH^{n,t}(B,B)$. In this context we recall that a set of vector in a Banach space is a \textit{Schauder basis} if every vector can be written uniquely as a possibly infinite linear combination of vectors from the set. The distinction to a `normal' basis of a vector space will plays a role when studying the Hochschild cohomology of proper graded gentle algebras which are not homologically smooth.

\subsection{Gerstenhaber bracket and cup product}
As before we follows the conventions of \cite{RoitzheimWhitehouse}. We will use the short hand notation $\1=\operatorname{Id}_A$. For $f \in C^{m,r}(B)$, $g \in C^{n,s}(B)$ an $1 \leq i \leq m$, one defines
\begin{displaymath}
f \circ g \coloneqq  \sum_{i=1}^{m} (-1)^{(i-1)(n-1)} f \circ_i g,
\end{displaymath}
\noindent where
\begin{displaymath}
f \circ_i g \coloneqq  (-1)^{(m-1)(n+s-1)} f\big(\1^{\otimes (i-1)} \otimes g \otimes \1^{\otimes (n-i)} \big).
\end{displaymath}
Here, the relevant degree for the Koszul sign rule is the degree of homogeneity $s$ of $g$. Then one takes the graded commutator of $\circ$ with respect to this shifted total degree,  that is, the degree on $C(A)[1]$, which yields

\begin{displaymath}
\begin{aligned}
[f,g] & = f \circ g - (-1)^{(m+r-1)(n+s-1)} g \circ f.
\end{aligned}
\end{displaymath}
\noindent Together with the previous differential, $[-,-]$ endows  $C^{\bullet}(A)[1]$ with the structure of a dg Lie algebra. The induced graded Lie bracket on the shifted cohomology $\HHH^{\bullet +1}(B,B)$ is the usual Gerstenhaber bracket. Likewise, the cup product on Hochschild cohomology stems is induces from the chain level operation
\begin{displaymath}
f \cup g \coloneqq m_2(f \otimes g).
\end{displaymath}
Together with the previous differential, it endows $C(A)$ with the structure of a dg algebra.

\subsection{The bigraded Bardzell resolution} Throughout the section, we fix a graded gentle quiver $(Q,I)$ and its associated graded gentle algebra $A$.
For $n \geq 1$, let $\cP_n$ denote the set of anti-paths of $A$ of length $n$. In particular, $\cP_1=Q_1$ and $\cP_0 \coloneqq Q_0$. Degrees of paths endow $\Bbbk \cP_n$ with a grading which we will call the \textbf{internal grading}. In this way $\Bbbk \cP_n$ naturally becomes a graded bimodule over the semi-simple ring $E\coloneqq \Bbbk \cP_0$ and $A \otimes_{E} \cP_n \otimes_{E} A$ inherits the structure of a graded $A$-bimodule with respect to the induced internal grading. The \textbf{bigraded Bardzell resolution} of $A$ is the complex $\cR_{\bullet}=(\cR_n)_{n \geq 0}$ of graded $A$-bimodules with $$\cR_n= A \otimes_{E} \Bbbk \cP_n \otimes_{E} A,$$. Its differential $d_{n}\colon\cR_n \rightarrow \cR_{n-1}$ is determined for an anti-path $u=\alpha_1 \dots \alpha_n$, $\alpha_i \in Q_1$, by
\begin{displaymath}
	d_n(1 \otimes_E u \otimes_E 1) = \begin{cases} \alpha_1 \otimes_E \alpha_2 \cdots \alpha_n \otimes_E 1 + (-1)^n 1 \otimes_E \alpha_1 \cdots \alpha_{n-1} \otimes_E \alpha_n, & \text{if $n \geq 1$}; \\ \alpha_1 \otimes_E e_{t(\alpha_1)} \otimes_E 1 - 1 \otimes_E e_{s(\alpha_1)} \otimes_E \alpha_1, & \text{if $n=1$}; \end{cases} 
\end{displaymath}
\noindent We denote by $A^e$ and $E^e$ the respective enveloping algebras over $\Bbbk$. Then one has an isomorphism of graded $A$-bimodules 
\begin{equation}\label{eq: identification bimodules and E-linear maps}
	\begin{tikzcd}[row sep=0.5em]
		\varphi\colon\Hom_{E^e}(\Bbbk \cP_n, A) \arrow{rr}{\sim} && \Hom_{A^e}(\cR_n, A), \\
		f \arrow[mapsto]{rr} && \Big( a \otimes u \otimes b \mapsto (-1)^{|f||a|} a f(u) b \Big),
	\end{tikzcd}
\end{equation}
\noindent so that $\varphi(f)=m_A(m_A(\mathbf{1} \otimes f) \otimes \mathbf{1})=m_A(\mathbf{1} \otimes m_A(f \otimes \mathbf{1}))$. Consequently, $$C_{\cP}^{\bullet}(A)\coloneqq \big(\Hom_{E^e}(\Bbbk \cP_n, A)\big)_{n \geq 0}$$ is a cochain complex of graded $A$-bimodules. It is standard to identify $C_{\cP}^n(A)$ with the vector space spanned by the set of parallel paths $$\left(\cP_n \parallel Q\right)\coloneqq\{ (p, q) \mid  p \in \cP_n, q \not \in I \text{ path}\colon s(\alpha)=s(\beta) \wedge t(\alpha)=t(\beta)\}.$$
\noindent Explicitly, $(p, q) \in (\cP_n \parallel Q)$ is identified with $p^{\ast} \cdot q$, where $p^{\ast} \in \Hom_{E^e}(\Bbbk\cP_n, \Bbbk)$ denotes the dual of $p$ with respect to the basis $\cP_n$. In particular, $(p, q)$ corresponds to an element of bidegree $(l(p), |q|-|p|)$ in $C_{\cP}(A)$ with the first entry corresponding to the cohomological and the second corresponding to the internal degree. With this, the differential of $C_{\cP}(A)$ is given as follows. For $f=(p, q) \in \left(\cP_n \parallel Q\right)$, set $n=l(p)$ and $s\coloneqq |f|=|q| - |p|$. Then
\begin{equation}\label{eq: Bardzell differential}
	d(p, q)= d_L(p,q) + d_R(p,q) \coloneqq (-1)^{s} \Big( \sum_{\alpha} (-1)^{s |\alpha|} (\alpha p, \alpha q) - \sum_{\beta} (-1)^n (p \beta, q\beta) \Big),
\end{equation}
where the left (resp.~right) sum is indexed by those arrows $\alpha \in Q_1$ (resp.~$\beta \in Q_1$) such that $\alpha p$ (resp.~$p \beta$) lies in $\cP_{n+1}$ and $\alpha q \not \in I$ (resp.~$q\beta \not \in I$).  We note that if $n \geq 1$, there exists at most one such $\alpha$ (resp.~$\beta)$ due to the nature of relations in a gentle quiver. The term $d_L(p,q)$ (resp.~$d_R(p,q)$) corresponds to the first (resp.~second) sum in \eqref{eq: Bardzell differential} including the sign $(-1)^s$ (resp.~$- (-1)^{s}$)).

\subsection{Comparison maps between Hochschild and Bardzell complex} Analogous to the ungraded case in \cite[Section 4]{RedondoRoman} one has quasi-inverse homotopy equivalences of complexes of graded $A$-bimodules
\begin{displaymath}
	\begin{tikzcd}
		\cR_{\bullet} \arrow[shift left=0.4em]{rr}{U} && \cB_{\bullet} \arrow[shift left=0.4em]{ll}{V},  &  V \circ U=\operatorname{Id}_{\cR_{\bullet}},
	\end{tikzcd}
\end{displaymath}
\noindent where $\cB_{\bullet}=(\cB_n)_{n \geq 0}$ denotes the graded bar resolution. The maps $U$ and $V$ induces maps $U^{\ast}$ and $V^{\ast}$ between $C(A)$ and $C_{\cP}(A)$ in opposite directions with $U^{\ast} \circ V^{\ast}=\operatorname{Id}_{C_{\cP}(A)}$ and exhibit $C_{\cP}(B)$ as a retract of $C(B)$. We recall that $\cB_n=A^{\otimes_E (n+2)}$ and the differential is given by $$d_n(a_0 \otimes \cdots \otimes a_n)=\sum_{i=0}^{n}(-1)^i \mathbf{1}^{\otimes_E i} \otimes_E m_A \otimes_E \mathbf{1}^{\otimes (n-1-i)},$$
where $m_A\colon A^{\otimes 2} \rightarrow A$ denotes the multiplication map. The maps $U$ and $V$ are defined exactly as in the ungraded case, that is, for $v=1 \otimes v_1 \otimes \cdots \otimes v_n \otimes 1$ and $n \geq 2$, one has
\begin{equation}\label{EquationFormulaComparisonMap}
	V(v) = \begin{cases}v' \otimes \alpha \otimes v'', & \text{if $v_1 \cdots v_n=v'\alpha v''$, where $\alpha \in \cP_n$, $v_i, v',v'' \not \in I$;} \\ 0, & \text{otherwise.}  \end{cases}
\end{equation}
as well as
\begin{displaymath}\label{EquationFormulaComparisonMap}
	\begin{aligned}
		V^1(1 \otimes v \otimes 1) \coloneqq & \begin{cases}0, & \text{if $l(v)=0$;} \\ \sum_{i=1}^{l(v)}\alpha_{1} \cdots \alpha_{i-1} \otimes \alpha_i \otimes \alpha_{i+1} \cdots \alpha_n, & \text{if $v=\alpha_{1} \cdots \alpha_n$, $\alpha_j \in Q_1$.} 
		\end{cases} \\ \\
		V^0(1 \otimes x \otimes 1)= & 1 \otimes x \otimes 1,
	\end{aligned}
\end{displaymath}
\noindent for $x \in Q_0$ and under the identification $A^{\otimes_E 2}\cong A \otimes_E E \otimes_E A$. The map $U$ is easier to describe and for $u=\alpha_1 \cdots \alpha_n \in \cP_n$, $n \geq 1$, is given by
\begin{displaymath}
	U(1 \otimes u \otimes 1)=1 \otimes \alpha_1 \otimes \cdots \otimes \alpha_n \otimes 1,
\end{displaymath} 
\noindent and $U_0(1 \otimes x \otimes 1)= 1 \otimes x \otimes 1$ for all $x \in Q_0$, again under the above identifications. Taking into account the Koszul sign rule, one recovers  formula \eqref{eq: differential Bar resolution} after applying $\Hom_{A^e}(-, A)$ and the identification of morphisms between free $A$-bimodules with certain $E$-linear homomorphisms\footnote{The sign arising upon evaluation from $m_B(\mathbf{1} \otimes f)$ in \eqref{eq: differential Bar resolution} is a consequence of this identification.} as in \eqref{eq: identification bimodules and E-linear maps}.

\subsection{Induced operations on the Bardzell resolution} The maps $U^{\ast}$ and $V^{\ast}$ allow to pull back binary operations $\star$ such as $\circ_i$, $[-,-]$ and $\cup$ from $C(A)$ to operations on $C_{\cP}(A)$ via the formula $$f \star g \coloneqq U^{\ast}\big(V^{\ast}(f) \star V^{\ast}(g)\big),$$
\noindent where $f, g \in C_{\cP}(A)$. We now give explicit formulae for the induced operations $\circ_i$ and $\cup$ on $C_{\cP}(A)$ which generalise those in \cite[Section 4.2]{RedondoRoman}.

For a more compact notation, we will choose the elementary tensors of the form $q_1 \otimes \cdots q_n$, where $q_1, \dots, q_n \not \in I$ are composable paths in $Q$, as a basis of the vector space $A^{\otimes_E n}$. If $q$ is a path with $t(q)=t(q_1)$ and $s(q)=s(q_n)$, we write $(q_1 \otimes \cdots \otimes q_n ,q)$ for the dual element in $C^{n, \bullet}(A)$ which maps $q_1 \otimes \cdots q_n$ to $q$ and all other basis elements of $A^{\otimes_E n}$ to zero. Every element in $C^{n,\bullet}(A)$ is a (possibly infinite) sum of such dual elements. The necessity for infinite sum occurs when $A$ is infinite-dimensional as $C^{n, \bullet}(A)$ is an infinite product of copies of $A$ in this case.  In what follows, we will consider a pair $(u,v)$ with $u$ and $v$ paths in $Q$ as the zero element in $C_{\cP}(A)$ is $u$ is not a well-defined path, not an anti-path, or $v=0$ in $A$. For homogeneous $(p_1, q_1), (p_2, q_2) \in C_{\cP}(A)$ this means for example that we understand $(p_1p_2, q_1q_2)=0$ if $s(p_1) \neq t(p_2)$, $\alpha_m \beta_1 \in I$, or, if $q_1 q_2 \in I$.  We note that the total degree of $(p, q) \in C_{\cP}(A)$ is $|(p,q)| \coloneqq l(p) + |q| - |p|$.\medskip

\noindent  We note that for any homogeneous $(p,q)=(\alpha_1 \cdots \alpha_n, q ) \in C_{\cP}^{n,s}(A)$, $n \geq 2$, \begin{displaymath}
V^{\ast}((p,q)) = \sum_{u, v} (u \alpha_1 \otimes \alpha_2 \otimes \cdots \otimes \alpha_{n-1} \otimes \alpha_n v , uqv),
\end{displaymath}
\noindent where the (possibly infinite) sum is indexed by all (possibly trivial) paths $u, v$ such that $u \alpha_1, \alpha_n v \not \in I$ and $u q v \not \in I$. For $n=1$, we have $V^{\ast}((p,q))=(\alpha_1 \otimes \cdots \alpha_n ,q)$ and also $V^{\ast}((x, q))=(x, q)$ for $x \in Q_0$. $U^{\ast}$ maps an element $(q_1 \otimes \cdots \otimes q_m, q) \in C^{m, \bullet}(A)$, $m \geq 1$ to zero unless $q_1, \dots, q_{m}$ are arrows and $q_1 \cdots q_m$ is an anti-path. Finally, $U^{\ast}((x, q))=(x,q)$, if $x \in Q_0$.\medskip

\noindent The operations $\cup$ and $\circ_i$ on $C_{\cP}(A)$ are now given as follows. Explicitly, for $(p_1, q_1) \in C_{\cP}^{m,r}(A)$ and $(p_2, q_2) \in C_{\cP}^{n,s}(A)$ with $p_1=\alpha_1 \cdots \alpha_m$ and $p_2=\beta_1 \cdots \beta_n$, $\alpha_i, \beta_j \in Q_1$, and for $1 \leq i \leq m$, one obtains
\begin{equation}\label{eq: cup product}
(p_1, q_1) \cup (p_2, q_2) = (-1)^{(|q_2|-|p_2|)|p_1|}(p_1 p_2, q_1 q_2),
\end{equation} 
For $\circ_i$ one has four cases. First,

\begin{equation}\label{eq: composition product 1}
(p_1, q_1) \circ_i (p_2, q_2)= (-1)^{(m-1)(n+s-1) + |\alpha_1 \cdots \alpha_{i-1}|(|q_2|-|p_2|)} (\alpha_1 \cdots \alpha_{i-1}p_2 \alpha_{i+1} \cdots \alpha_m, q_1)
\end{equation}
\noindent whenever $1 < i < m$ and $q_2=\alpha_i$. Secondly,
\begin{equation}\label{eq: composition product 2}
(p_1, q_1) \circ_1 (p_2, q_2)= (-1)^{(m-1)(n+s-1)} (p_2 \alpha_2 \cdots \alpha_m, uq_1) 
\end{equation}
\noindent if $l(p_1) \geq 2$, $l(q_2) \geq 2$ and $q_2=u\alpha_1$ for some path $u$ which necessarily does not lie in $I$. The next case is
\begin{equation}\label{eq: composition product 3}
(p_1, q_1) \circ_{m} (p_2, q_2)= (-1)^{(m-1)(n+s-1) + |\alpha_1 \cdots \alpha_{m-1}|(|q_2|- |p_2|)} (\alpha_1 \cdots \alpha_{n-1} p_2, q_1 u), 
\end{equation}
\noindent if $l(p_1) \geq 2$, $l(q_2) \geq 2$ and $q_2=\alpha_m u$ for some path $u$ which necessarily does not lie in $I$. The last case is a mix of the previous two, when $l(p_1)=1$ so that $m=1$ in which case 
\begin{equation}\label{eq: composition product 4}
(p_1, q_1) \circ_1 (p_2, q_2)= (-1)^{(m-1)(n+s-1)} (p_2, uq_1v) 
\end{equation}
if $q_2=u\alpha_1 v$ for some paths $u, v$. All expressions $f \circ_i g$ which do not fall in any of the above cases are zero.

\subsection{Computations of the Hochschild cohomology}		

\noindent We now give a description of the Hochschild cohomology of a graded gentle algebra as a vector space. Our proof is based on the graph associated to the Bardzell differential which appears naturally from the perspective of discrete Morse theory, cf.~\cite{Tamaroff}. This allows us to streamline combinatorial arguments using the combinatorics of this graph. 

\subsubsection{The graph and canonical decompositions}
Let $\overline{G}$ denote the quiver with vertices $\overline{G}_0 \coloneqq \{0\} \sqcup \bigsqcup_{n \geq 0}\left( \cP_n \parallel Q\right)$ which has an arrow $x \rightarrow y$ labelled with $L$ (resp.~$R$) whenever $y$ appears with non-trivial coefficient when expressing $d_L(x)$  (resp.~$d_R(x)$) as a linear combination of elements from $\overline{G}_0 \setminus \{0\}$. We also add an arrow $x \rightarrow 0$ labelled $L$ (resp.~$R$) if $d_L(x)=0$ (resp.~$d_R(x)=0$). We denote by $G \subseteq \overline{G}$ the full subquiver spanned by $\{0\}$ and all vertices $(p,q)$ such that $l(p)+l(q)\geq 1$, that is, we exclude vertices of the form $(x,x)$, $x \in Q_0$, in $G$. We also consider the full subquivers $G^{\star} \subseteq G$ and $\overline{G}^{\star} \subseteq G$ spanned by the non-zero vertices in each case. We record a few elementary observations:

\begin{enumerate}
	\item Arrows in $\overline{G}$ are uniquely determined by their target and label.
	\item $\overline{G}$ has no oriented cycles and its arrows define a partial order on its set of vertices with $x < y$ if there exists a path from $x$ to $y$. It has $0$ as its largest element and each $(x,x)$, $x \in Q_0$, is a minimal element. The 'height maps' $h\colon \overline{G}_0 \setminus \{0\} \rightarrow \mathbb{N}$, $(p,q) \mapsto l(p)$ and $h'\colon \overline{G}_0 \setminus \{0\} \rightarrow \mathbb{N}$, $(p,q) \mapsto l(q)$  are order preserving.  
	\item The quiver $G$ is a binary tree. More precisely, every vertex $v=(p,q)$ of $G$ is the source of exactly two arrows, one with label $L$ and the other with label $R$. Every vertex in $\overline{G}^{\star}$ is the target of at most two arrows and if there are two, then they have different labels. Vertices $(x,x)$, $x \in Q_0$ are source of at most $4$ arrows with at most $2$ of each label.
	
	\item The passage from $x$ to $d_L(x)$ and $d_R(x)$ induces endomorphisms of partially ordered sets $L, R \colon G_0 \rightarrow G_0$ with $L^2=R^2$ equal to the constant map $G_0 \rightarrow \{0\}$ and $R \circ L=L \circ R$. More generally, for any subset  $U \subseteq \overline{G}_0$ let $L(U)$ (resp.~$R(U)$) denote the set of  vertices $y$ which appear as targets of an arrow  with label $L$ and source in $U$. Then, $L^2(U)=\{0\}=R^2(U)$ and $L(R(U))=R(L(U))$ for all $U \subseteq \overline{G}_0$.\\[0.5em]
	
	\item If $v \in \overline{G}_0$ such that $0 \neq u \in L(v) \cap R(v)$, then $v=(x,x)$ for some $x \in Q_0$ and there exists a loop $\alpha \in Q_1$ with $s(\alpha)=t(\alpha)=x$. In particular, if $v \in G_0$, then $\{L(v), R(v)\}$ is linearly independent or contains $0$.
\end{enumerate}

\begin{definition}
For $v \in \overline{G}^{\star}$, a vertex $w$ such that $v \in L(w) \cup R(w)$ is called a \textbf{parent} of $v$ and $v$ is a \textbf{child} of $w$. Thus, any two vertices have at most one common child and every vertex has at most two parents. Parents who have a common child be called \textbf{neighbours}. This defines an equivalence relation on vertices whose equivalence classes we call \textbf{neighbourhoods}. Thus, a neighbourhood is a subset $V \subseteq \overline{G}^{\star}$ which contains the neighbours of each of its elements. The neighbourhood of a subset $V \subseteq \overline{G}^{\star}$ is the smallest neighbourhood which contains $V$. A neighbourhood is \textbf{atomic} if it coincides with the neighbourhood of each of its elements. 
\end{definition}

\noindent We use the properties (1)-(5) above to construct a special decomposition of every non-zero $f \in C_{\cP}(A)$. Write $f$ uniquely as a finite sum $f=\sum_{j \in J} \lambda_j (p_j, q_j)$ such that $\lambda_j \in \Bbbk^{\times}$ and $(p_i, q_i) \neq (p_j, q_j)$ for all $i \neq j \in J$.  Then we define $V(f)\coloneqq \{(p_j, q_j) \mid j \in J\} \subseteq \overline{G}^{\star}$ and refer to its cardinality $|J|$ as the \textbf{rank} of $f$. In particular,  for all $f \in C_{\cP}(A)$, $V(d(f)) \subseteq L(V(f)) \cup R(V(f))$. For $v=(p_j, q_j) \in V(f)$ define $\lambda_f(v)\coloneqq\lambda_j$. Any partition $V(f)=\bigsqcup_{x \in \chi} V_x$, $V_x \neq \emptyset$ yields a sum decomposition $f=\sum_{x \in \chi} f_x$, $f_x= \sum_{j \in V_x}\lambda_j (p_j, q_j)$. Define $\overline{G}(f)$ as the full subquiver spanned by the vertices in $V(f)$ as well as their children.  The intersection of $\overline{G}(f)$ with $G$ and $G^{\star}$ defines full subquivers $G(f)$ and $G^{\star}(f)$ of $G$ and $G^{\star}$ respectively. The set $V(f)$ decomposes into disjoint level sets $V_{n,n'}(f)\coloneqq h^{-1}(n) \cap h^{\prime -1}(n') \cap V(f)$.  Another partition arises from the connected components of the underlying graph of $G^{\star}(f)$. By abuse of language we refer to the corresponding summands of $f$ as its \textbf{connected components}. We then say that $f$ is \textbf{connected} if it has a single connected component and it is \textbf{levelled} if $V(f)=V_{n,n'}(f)$ for some pair $n, n' \in \mathbb{N}$. By decomposing each non-empty $V_{n,n'}(f)$ into its connected components, we obtain a partition of $V(f)$ and a sum decomposition of $f=\sum_{j \in J}f_j$ such that each $f_j$ is connected and levelled. We call this the \textbf{canonical partition} and \textbf{canonical decomposition} respectively. The elements $f_i$ will be called the \textbf{atoms} of $f$ and $f$ is \textbf{atomic} if it is the sum of a single atom. Immediate consequences of these definitions are the following.

\begin{lem}If $f \in C_{\cP}(A)$ is levelled, then so are $d(f), d_L(f)$ and $d_R(f)$ whenever they are non-zero. If $f$ is atomic, then so is the neighbourhood of $V(f)$.
\end{lem}
\noindent Now, one also easily verifies the following property.

\begin{lem}\label{lem: differentials atomic element linearly independent}Let $0 \neq f \in C_{\cP}(A)$ and let $V(f)=\bigsqcup_{j \in J}V_j$ be the canonical partition. Then for all $i \neq j \in J$, $$\Big(L(V_i) \cup R(V_i)\Big) \cap \Big(L(V_j) \cup R(V_j)\Big)= \emptyset.$$ 
	In other words, $V_i$ and $V_j$ have no common children. \end{lem}
\noindent This leads to an equivalent characterisation of the canonical decomposition as the unique decomposition $f=\sum_{j \in J}f_j$ into connected, levelled elements $f_j$ so that $G^{\star}(f_j) \cap G^{\star}(f_i)=\emptyset$ for all $i \neq j \in J$. Consequently, an element is atomic if and only if it is levelled and connected. Being subsets of a basis of $\Bbbk \left(\cP \parallel Q\right)$, it follows that the two subsets in \Cref{lem: differentials atomic element linearly independent} are linearly independent which implies the following.
\begin{cor}\label{cor: differential determined by differential of atoms}
	Let $f \in C_{\cP}(A)$. Then $d(f)=0$ if and only if $d(a)=0$ for each atom $a$ of $f$.
\end{cor}
\noindent Another straightforward consequence is the compatibility of canonical decompositions with the differential.
\begin{lem}\label{cor: behaviour of canonical decompositions under differential}
	Let $f=\sum_{t \in T}f_t$ and $u=\sum_{s \in S}u_s$ be canonical decompositions of elements $f, u \in C_{\cP}(A)$. If $f=d(u)$, then there exists a unique function $\varphi\colon T \rightarrow S$ such that for every $s \in S$, $$d(u_s)=\sum_{t \in \varphi^{-1}(s)} f_t.$$
\end{lem}
\begin{proof}
	It follows from \Cref{lem: differentials atomic element linearly independent} that $L \circ R(V(u_s)) \cap L \circ R(V(u_{s'})) \subseteq \{0\}$ for all $s \neq s' \in S$. Since $R \circ L(U)=L \circ R(U)$ for any $U \subseteq V(\overline{G})$ and the equivalent characterisation of the canonical decomposition, it then follows that each connected component of the levelled element $d(u_s)$ is in fact an atom of $f$. The function $\varphi$ is then simply defined such that $t \in \varphi^{-1}(s)$ whenever $f_t$ appears as an atom of $d(u_s)$.
\end{proof}

\begin{cor}\label{cor: relations in cohomology between atomic cocycles}
	Let $f_1, \dots, f_n \in C_{\cP}(A)$ be atomic cocycles such that $f=\sum_{i=1}^nf_i$ is a coboundary and such that each $f_i$ is an atom of $f$. Then, there exists a partition $\{0, \dots, n\}=\bigsqcup_{j \in J} S_j$ such that for each $j \in J$, $\sum_{s \in S_j} f_s=d(u)$ for some atomic $u$.
\end{cor}
\noindent This will be useful to study the relations between cocycles of $C_{\cP}(A)$ inside Hochschild cohomology.

\subsubsection{Classification of atomic cocycles} \ \medskip

\noindent \Cref{cor: differential determined by differential of atoms} allows us to reduce the analysis of cocycles in $C_{\cP}(A)$ to the atomic ones. A basis of the Hochschild cohomology of a graded gentle algebra then appears naturally by classifying atomic cocycles and relations in cohomology among them. The proofs are spread across a series of lemmas.

\begin{lem}\label{lem: atomic coycles are of type A or tilde A}Let $f \in C_{\cP}(A)$ be atomic. If $V(f) \not \subseteq G$, then $V(f)=\{(x,x)\}$ for some $x \in Q_0$. If $V(f) \subseteq G$, then $G^{\star}(f)$ is a finite quiver of type $A$ or type $\tilde{A}$. If it is of type $\tilde{A}$, then $G(f)=G^{\star}(f)$ and if it is of type $A$, then $G(f)$ is a quiver of type $\tilde{A}$ with one additional vertex $\{0\}=G(f) \setminus G^{\star}(f)$. In both cases $G(f)$ has $2|V(f)|$ arrows whose labels alternate between $L$ and $R$ in the natural cyclic order. 
\end{lem}
\begin{proof}
	If $V(f)$ does not lie in $G$, it must contain an element of the form $(x,x)$, $x \in Q_0$, and since $f$ is levelled, $V(f)=V_{0,0}(f) \subseteq \{(y,y) \mid y \in Q_0\}$. In what follows we assume $V(f) \subseteq G$. The finiteness of $G^{\star}(f)$ is clear. Since $f$ is connected and levelled, say $V=V_{n, n'}(f)$, it follows that $G^{\star}(f)$ is connected and $L(V(f)) \cup R(V(f)) \subseteq \{0\} \sqcup h^{-1}(n+1) \cap h^{\prime -1}(n'+1)$.  By definition, each arrow in $G(f)$ connects a source in $V(f)$ to a target in $\{0\} \sqcup h^{-1}(n+1) \cap h^{\prime -1}(n'+1)$ so that all vertices of  $G^{\star}(f)$ lie in $V_{n,n'}(f) \sqcup V_{n+1, n'+1}(f)$. As we remarked earlier, each  $v \in G^{\star}(f)$ is the target of at most one arrow in $G^{\star}(f)$  for each of the two labels $L$ and $R$ and each such arrow has source in $V(f)$. Moreover, each $w \in V(f)$ is the source of exactly two arrows in $G(f)$, again one for each label. Set $r \coloneqq |V(f)|$. Because $f$ is connected and arrows with non-zero target are determined by their target and label, we now conclude that either $0 \in G(f)$ and then $G^{\star}(f)$ is a quiver of type $A$ with $2(r-1)$ arrows, or, $G(f) \subseteq G^{\star}$ and then $G(f)$ of type $\tilde{A}$ with $2r$ arrows. The remaining claims are a direct consequence of this.
\end{proof}
\noindent The \textbf{type} of an atomic cocycle $f$ with $V(f) \subseteq G$ is the type of the quiver $G^{\star}(f)$. As we saw in the proof of the previous lemma, the two possible cases are distinguished by whether $0 \in G(f)$ or not. Moreover, by \Cref{lem: atomic coycles are of type A or tilde A} one obtains a canonical linear order (type $A$) or cyclic order (type $\tilde{A}$) on $V(f)$ by requiring that the successor $w \in V(f)$ of an element $v \in V(f)$ satisfies $R(v)=L(w)$. 
\begin{lem}\label{lem: atomic cocycle up to sign}
	Let $f$ be an atomic cocycle with $V(f) \subseteq G_0$. Then for all $v, v' \in V(f)$, $$\lambda_f(v)=\sigma(v,v')\lambda_f(v')$$ for some $\sigma(v,v') \in \{\pm 1\}$ which only depends on $v, v' \in G$.
\end{lem}
\begin{proof}
	Let $w$ be the successor of a vertex $v \in V(f)$ and hence $L(w)=R(v) \neq 0$. Because $v, w \in G_0$, $L(v)=0$ (resp.~$R(w)=0$) or $\{L(v), R(v)\}$ (resp.~$\{L(w), R(w)\}$) is linearly independent. Thus, $d(f)=0$ implies that $d_L( \lambda_f(w) w)=-d_R(\lambda_f(v) v)$.
	The result then follows by induction and explicit formulas for $\sigma(v,v')$ can be deduced from the signs appearing in $d_L$ and $d_R$ in \eqref{eq: Bardzell differential}.
\end{proof}
\begin{cor}\label{cor: atomic cocycle determined by its graph}
	Let $f$ be an atomic cocycle with $V(f) \subseteq G_0$.	Then, up to scalar multiple, $f$ is uniquely determined by $G^{\star}(f)$ (resp.~$G(f)$). 
\end{cor}
\begin{lem}\label{cor: trace elements type tilde A}
	Let $U \subseteq G^{\star}$ be a subquiver of type $\tilde{A}$ such that $U \subseteq h^{-1}(\{n,n+1\})$ for some $n \geq 0$. If $U=G^{\star}(f)$ for some atomic cocycle $f$, then exactly one of the following is true:
	
	\begin{enumerate}
		\item There exists a complete anti-path $p=\alpha_1 \cdots \alpha_n \in \cP_n$ of period $r$ such that $f$ is a non-zero scalar multiple of
		\begin{displaymath}
			\cN^0(p) \coloneqq \sum_{i=0}^{r-1} (-1)^{i n + |p| |\alpha_1 \cdots \alpha_i|} \big(\rot^i(p), s(\rot^i(p))\big)
		\end{displaymath}
		
		\item  There exists a  complete path $q=\beta_1 \cdots \beta_n$, $l(q)=n$, of period $r$ such that $f$ is a non-zero scalar multiple of
		\begin{displaymath}
			\cN^0(q) \coloneqq \sum_{i=0}^{r-1} (-1)^{ |q| |\beta_1 \cdots \beta_i|} \big(s(\rot^i(q)), \rot^i(q)\big).
		\end{displaymath}
	\end{enumerate}
\end{lem}
\begin{proof}First we assume that $h(f)\geq 1$ and let $w \in V(f)$ be the successor of an element $v=(p,q) \in V(f)$ with $p=\alpha_1 \cdots \alpha_n \in \cP_n$. By the formulas for $d_L$ and $d_R$ we infer that $w=(\alpha_2 \cdots \alpha_n \beta, q' \beta)$, where $\beta \in Q_1$ is an arrow such that $q=\alpha_1 q'$, $\alpha_n \beta \in I$ and $q' \beta \not \in I$. By applying this formula repeatedly and using the assumption that $(Q,I)$ is gentle, the fact that $G^{\star}(f)$ is a cycle forces $\beta=\alpha_1$. By induction we therefore find that $p$ is complete. We claim that $l(q')=0$ and hence $l(q)=1$. Indeed, following the sequence of successors we find that $pq'=q'p$ in $A$ which implies $l(q')=0$ since $l(p)=h(f) > 0$ since $I$ is generated only by zero relations. We conclude that $V(f)=\{(\rot^i(p), s(\rot^i(p)) \mid 0 \leq i \leq r-1\}$ where $r$ denotes the period of $p$. The concrete signs of $\cN^0(p)$ then follows by tracking the explicit signs in \Cref{lem: atomic cocycle up to sign}. The other case when $h(f)=l(p)=0$ and hence $l(q) > 0$ is treated with similar arguments and shows that $q$ is a complete path.
\end{proof}

\begin{cor}\label{cor: 0 trace not coboundary}
	Let $f$ be an atomic cocycle with $V(f) \subseteq G_0$ of type $\tilde{A}$. Then $f$ is not a coboundary.
\end{cor}
\begin{proof}
	The explicit formulas for $\cN^0(-)$ in \Cref{cor: trace elements type tilde A} show that either $h(f)=0$ or $h'(f)=0$. Thus, none of the elements in $V(f)$ are target of an arrow in $\overline{G}$.
\end{proof}
\noindent The following lemma describes the differentials of the expressions $\cN^0(p)$ and $\cN^0(q)$ from \Cref{cor: trace elements type tilde A} which can be defined regardless of whether they are cocycles or not. 

\begin{lem}\label{lem: differentials of atomic cocycles of type tilde A}Let $p=\alpha_1 \cdots \alpha_n \in \cP_n$ and $q=\beta_1 \cdots \beta_m$ a complete path. Then
	\begin{displaymath}
		\begin{aligned}
			d(\cN^0(p)) & = (-1)^{|p|}\cdot \big(1- (-1)^{\omega(p)} \big) (\alpha_1 p, \alpha_1), \\
			d(\cN^0(q)) & =\big(1- (-1)^{\omega(q)}\big) (\beta_n, \beta_n q).
		\end{aligned}
	\end{displaymath}
	\noindent where as before $\omega(p)= l(p) - |p|= n - |p|$ and $\omega(q) = -|q|$.
\end{lem}
\begin{proof}
Verified by direct computation which we omit.
\end{proof}

\noindent Next, we treat atomic cocycles of type $A$. The same arguments as in $\tilde{A}$-case show that up to scalar, an atomic cocycle $f$ with $V(f) \subseteq G_0$ of type $A$ is uniquely determined by $G^{\star}(f)$, or equivalently, by $G(f)$. Note that in contrast to the $\tilde{A}$-case, there are no obstructions to attaching an atomic cocycle to a given suitable quiver of type $A$.

\noindent We are left with classifying atomic coycles of type $A$. 
\begin{lem}\label{lem: atomic cocycle type A coboundary or rank 1}
	Let $f$ be an atomic cocycle with $V(f) \subseteq G_0$ of type $A$. Then $f$ has rank $2$ and is a coboundary, or it has rank $1$.
\end{lem}
\begin{proof}
	The graph $G^{\star}$ has the property that whenever $v, v' \in G^{\star}$ are such that $R(v)=L(w) \in G^{\star}$, then there exists a unique $w \in G^{\star}$ such that $L(w)=v$ and $R(w)=v'$. Now if $f$ has rank $r \geq 2$, there exists $v \in V(f)$ which has a successor $v' \in V(f)$. Since $R(v)=L(v')$ by definition, there exists a unique $w \in G^{\star}$ such that $L(w)=v$ and $R(w)=v'$. In particular, $L(v)=L^2(w)=0=R^2(w)=R(v')$ and hence $r=2$ because $f$ is  connected. The fact that $d(f)=0$, now also implies $d_R(\lambda_f(v)v)=-d_L(\lambda_f(v') v')$. On the other hand, we have $d^2(w)=0$ which implies $d_R (d_L(w))=-d_L(d_R(w))$ and which shows that $f$ is a multiple of $d(w)$ and therefore a coboundary.
\end{proof}

\begin{lem}\label{lem: special vertices in canonical partition}
	Let $v \in G^{\star}$ such that $L(v)=0=R(v)$. If $f \in C_{\cP}(A)$ is any cocycle such that $v \in V(f)$, then $\lambda_f(v) v$ is an atom of $f$.
\end{lem}
\begin{proof}
	Suppose $f=\sum_{j \in J}f_j$ is the canonical decomposition. Because $V(f_j) \cap V(f_i)$ for all $i \neq j \in J$, there exists by assumption a unique $j_0 \in J$ such that $v \in V(f_{j_0})$. But since $L(v)=0=R(v)$, $v$ is a sink of $G^{\star}(f_{j_0})$ and hence the subset $V(f_{j_0})$ of $\overline{G}^{\star}$ cannot be connected if it is strictly larger than $\{v\}$, so $V(f_{j_0})=\{v\}$ and thus $\lambda_f(v)v$ is an atom of $f$.
\end{proof}

\begin{lem}\label{lem: classification atomic cocycle type A rank 1}
	Let $f$ be an atomic cocycle of type $A$ and rank $1$. Then $f$ is a coboundary or exactly one of the following is true.
	
	\begin{enumerate}
		\item $f$ is a scalar multiple of $c= \cN^1(p) \coloneqq (\alpha_n p, \alpha_n)$ for some complete anti-path $p=\alpha_1 \cdots \alpha_n$. If $f$ is a coboundary, then $f=d(\lambda \cN^0(p))$ for some $\lambda \in \Bbbk^{\times}$. 
		\item $f$ is a scalar multiple of $c= \cN^1(q)\coloneqq (\beta_n, \beta_n q)$ for some complete path $q=\beta_1 \cdots \beta_n$. If $f$ is a coboundary, then $f=d(\lambda \cN^0(q))$ for some $\lambda \in \Bbbk^{\times}$. 
		\item $f$ is not a coboundary and is a scalar multiple of $c=(u, u_{\circ}) \in G^{\star}$ for some maximal anti-path $u$ with $u_{\circ}\neq \emptyset$.
		\item $f$ is not a coboundary and is a scalar multiple of $(s(u), u)$ for some closed maximal path $u$ of $A$. 
		\item $f$ is a scalar multiple of $c=(\alpha, \alpha)$, where $\alpha \in Q_1$.
	\end{enumerate}
	Conversely, any $c$ of the above shapes is an atomic cocycle.
\end{lem}
\begin{proof}To begin with, we observe that $L(c)=R(c)=0$ for all rank $1$ elements in $C_{\cP}(A)$ which are of the shape from (1)--(4). Hence such a $c$ is always a cocycle.
	Let $v=(p,q)$ be the unique vertex of $V(f)$ and write $p=\alpha_1 \cdots \alpha_n \in \cP_n$ and $q=\beta_1 \cdots \beta_m$, where $m=l(q)$. Because $f$ is a cocycle we have $L(v)=0=R(v)$. For $L$ this means that either there does not exist $\alpha \in Q_1$ such that $\alpha \alpha_1 \not \in I$  or, there exists such $\alpha$ but $\alpha \beta_1 \in I$. One deduces an analogous condition from $R(v)=0$.
	
	Now suppose $v$ has no parents. In particular, $f$ is not a coboundary. This happens if $l(p)=0$, $l(q)=0$, or we have, $\alpha_1 \neq \beta_1$ and $\alpha_n \neq \beta_m$. If $l(p)> 0$, then because of the nature of the relations in a gentle quiver, this means that $p$ cannot be extended to a longer anti-path on either of its ends and is hence maximal as well as $\para{p}\neq \emptyset$ which implies $q=\para{p}$. Likewise if $l(p)=0$, $L(v)=0=R(v)$ means that $q$ cannot be extended to a longer path of $A$. It follows that $f$ is of the shape in (3) or (4).
	
	Next, we assume that $v$ has parents, that is, it lies in the image of $L$ or $R$. Because the other case is analogous, we restrict ourselves to the case when $v=L(w)$ for some $w =(p', q') \in G^{\star}$. If $l(p')=l(q')=0$, then $v=(\alpha, \alpha)$ for some $\alpha \in Q_1$ with $s(\alpha)=s(q')=s(p')$ as claimed in (5). Otherwise, $w \in \overline{G}_0 \setminus \{0\}$ and $w$ is uniquely determined by $v$.
	
	From now on we may assume that $w \in G$, that is, $l(p')+l(q') \geq 1$. Then if $R(w)\neq 0$, one has $L(R(w))=R(L(w)=R(v)=0=R(R(w))$ so that $R(w)$ is again an atomic cocycle of rank $1$. Moreover, if $R(w) \neq 0$, we claim that $l(p')=0$ or $l(q')=0$. We have $R(w)=(p'\beta, q'\beta)$ for some $\beta \in Q_1$ such that $p' \beta \in \cP_n$. But $\alpha p'=p$ and hence $l(p')=0$ or $\alpha p' \beta \in \cP_{n+1}$. Therefore, if $l(p') > 0$, then $L(w)=R(v)=0$ just means that $\alpha q' \beta \in I$.  But because $L(w)=v \neq 0$, we know that $\alpha q' \not \in I$  and  because $R(w) \neq 0$, we know that $q' \beta \not \in I$. Since $I$ is generated by quadratic zero relations, this is possible only if $l(q')=0$. Now by repeatedly applying the same procedure that produced $R(w)$ from $v$ to the atomic cocycle $R(w)$ and so forth, we obtain an ordered sequence of vertices $v_0=v, v_1=R(w), v_2, \dots$ in $G^{\star}$ of the same height as well as vertices $w_0=w, \dots, w_s$ such that $L(w_i)=v_i$ and $R(w_i)=v_{i+1}$ for all $i \geq 0$. Either, this procedure terminates after finitely many steps with an element $v_s$ for which $R(v_s)=0$, or, it does not. In the finite case we observe that $d(v_s)=\pm L(v_s)$ and that therefore $d(\sum_{i=0}^{s-1} \lambda_i w_i)=f$ for appropriate scalars $\lambda_i \in \Bbbk^{\times}$ which shows that $f$ is a coboundary. In the infinite case, the finiteness of $Q$ implies that $p$ must be a complete anti-path if $l(q')=0$ or $q$ must be a complete path if $l(p')=0$. We therefore conclude in this scenario that $v$ is of the claimed shape in (1) and (2). Because the case of a complete path is analogous, we will now assume that $f$ is a multiple of $\cN^1(p)$ for some complete anti-path $p$. If $f$ is a coboundary, then by \Cref{cor: relations in cohomology between atomic cocycles} and because $\cN^1(p)$ is atomic, $f=d(g)$ for an atomic $g$. We finally note that $V(\cN^0(p))$ coincides with the neighbourhood of the parents of $f$ and hence $V(g) \subseteq V(\cN^0(p))$. Now the usual cancellation arguments and $d(g)=f$ force $g$ to be a scalar multiple of $\cN^0(p)$. 
\end{proof}

\begin{lem}\label{lem: linear dependencies between traces} Let $p$ be a complete path or anti-path. Then $\cN^0(p)$ and $\cN^0(\rot(p))$ are linear multiples of each other. Similarly, $\cN^1(p)$ and $\cN^1(\rot(p))$ are linear multiples of each other in cohomology. The assignment $\alpha \mapsto (\alpha, \alpha)$ induces an embedding into $\HHH^{1,0}(A,A)$ of the cokernel of the map
	$\Bbbk^{Q_0} \rightarrow \Bbbk^{Q_1}$, $(\lambda_x)_{x \in Q_0} \mapsto (\lambda_{t(\alpha)}-\lambda_{s(\alpha)})_{\alpha \in Q_1}$.
\end{lem}
\begin{proof}
	The first claim follows because $\cN^0(p)$ and $\cN^0(\rot(p))$ coincide up to sign. In case $p$ is a complete anti-path, one finds that $d((p, e_{s(p)}))$ is a signed sum of $\cN^1(p)$ and $\cN^1(\rot^{-1}(p))$ implying the second claim. The second claim is a version of \cite[Proposition 5.4]{OpperDerivedEquivalences}.
\end{proof}
\noindent The kernel of the map $\Bbbk^{Q_0} \rightarrow \Bbbk^{Q_1}$ is given by the image of the diagonal map and the elements $(\alpha, \alpha)$ span a subspace of $\HHH^{1,0}(A,A)$ which is isomorphic to $\HH^1(|Q|, \Bbbk) \cong \HH^1(\Sigma_A, \Bbbk)$. In particular, if $T \subseteq Q$ is a spanning tree, then $\{(\alpha, \alpha) \mid \alpha \in Q_1\}$ constitutes a basis of this subspace.
\begin{rem}The image of $\HH^1(\Sigma_A, \Bbbk)$ inside $\HHH^1(A,A)$ from \Cref{lem: linear dependencies between traces} is also related to the fundamental group of $A$ in the sense of \cite{MartinezVillaDeLaPenaFundamentalGroup} which can be generalised to graded algebras in the natural way.  In general, these groups depend on a chosen presentation of an algebra but is independent if the algebra is is monomial, e.g.~is a gentle algebra, in which case the fundamental group of $A$ is free abelian. There is a natural identification $\pi_1(A) \cong \pi_1(\Sigma_A)$ and following \cite{DeLaPenaSaorin}, for any ordinary (=ungraded) algebra $B$ there is a embedding $\Hom(\pi_1(B), \Bbbk) \hookrightarrow \HHH^1(B,B)$. If $A$ is concentrated in degree $0$, then under the identification $\Hom(\pi_1(\Sigma_A), \Bbbk) \cong \HH^1(\Sigma_A, \Bbbk)$, the embedding coincides with the embedding $\HH^1(\Sigma_A, \Bbbk) \hookrightarrow \HHH^{1,0}(A,A)$ induced by \Cref{lem: linear dependencies between traces}. The work \cite{BriggsRubioYDegrassiMaximalTori} shows that for every ordinary monomial algebra $B$, the rank of $\pi_1(B)$ agrees with the rank of any maximal torus in $\HHH^{1}(B, B)$.
\end{rem}

\begin{thm}\label{thm: appendix Hochschild basis}
	Let $(Q,I)$ be a graded gentle quiver, $A$ its associated graded gentle algebra which neither assumed to be homologically smooth nor proper. Choose
	\begin{itemize}
			\item  any spanning tree $T$ of $Q$,
		\item  any complete set of representatives $\cP^c$ (resp.~$\cQ^c$) of all rotation classes of complete anti-paths (resp.~paths) in $(Q,I)$.
	
	\end{itemize} Then for every $n \geq 0$ and every $d \in \mathbb{Z}$, the following elements constitute a $\Bbbk$-linear (Schauder) basis of $\HHH^{n,d}(A,A)$:
	\begin{enumerate}
		\item the unit $\mathbb{1}=\sum_{x \in Q_0}(x, x)$, if $(n,d)=(0,0)$;
		\item all elements of the form $\cN^s(p)$, $p \in \cP^c$, $s \in \{0,1\}$, such that $(n,d)=(l(p)+s, -|p|)$ and where $\omega(p) \in 2 \mathbb{Z} + s$ or $\operatorname{char} \Bbbk=2$.
		\item all elements of the form $\cN^s(q)$, $q \in \cQ^c$, $s \in \{0,1\}$, such that $(n,d)=(s,|q|)$ and where $\omega(q) \in 2 \mathbb{Z}+s$ or $\operatorname{char} \Bbbk = 2$.
		\item all elements of the form $(u, u_{\circ})$, where $u$ is a maximal anti-path with $u_{\circ}\neq \emptyset$, and $(n,d)=(l(u),|u_{\circ}|-|u|)$.
		\item all elements of the form $(s(u), u)$, where $u$ is a closed maximal path of of $A$ and $(n,d)=(0, |u|)$.
		\item all elements of the form $(\alpha, \alpha)$, $\alpha \in Q_1 \setminus T_1$, such that $(n,d)=(1,0)$.
	\end{enumerate}
\end{thm}
\begin{proof}Below we will refer to the elements appearing in (1)-(4) as the \textit{postulated basis}. Suppose $c$ is a complete path or complete anti-path. By \Cref{lem: classification atomic cocycle type A rank 1}, either the cocycle $\cN^1(c)$ is not a coboundary and then $\cN^0(c)$ is a cocycle (and not a coboundary), or $\cN^1(c)$ is not a coboundary and $\cN^0(c)$ is not a cocycle. So either both $\cN^0(c)$ and $\cN^1(c)$ contribute non-trivial Hochschild classes or neither of them. The conditions on $\omega(p)$ and $\omega(q)$ and the characteristic of $\Bbbk$ are a consequence of this observation.  More generally by \Cref{lem: differentials of atomic cocycles of type tilde A} and \Cref{lem: classification atomic cocycle type A rank 1}, each element of the postulated basis is an atomic cocycle but not a coboundary as follows from \Cref{cor: 0 trace not coboundary} and \Cref{lem: classification atomic cocycle type A rank 1}. By \Cref{cor: differential determined by differential of atoms}, every cocycle in $C_{\cP}(A)$ is a sum of atomic ones and by \Cref{lem: atomic coycles are of type A or tilde A}, every atomic cocycle if of type $A$ or $\tilde{A}$. It therefore follows from \Cref{cor: trace elements type tilde A}, \Cref{lem: atomic cocycle type A coboundary or rank 1}, \Cref{lem: classification atomic cocycle type A rank 1} and \Cref{lem: linear dependencies between traces} that the elements from (1)-(4) generate $\HHH^{n,s}(A,A)$ as a vector space. It remains to show that they are linearly independent. Let $f_1, \dots, f_n \in C_{\cP}(A)$ be a sequence of pairwise distinct elements of the postulated basis. Then $f=\sum_{i=1}^n \lambda f_i$ is a cocycle and we note that each $\lambda_i f_i$ is necessarily an atom of $f$.
	
 Suppose for the remainder of the proof that $f$ is a coboundary. Then by \Cref{cor: behaviour of canonical decompositions under differential} we may assume that $f=d(u)$ for some atomic $u \in C_{\cP}(A)$. Consequently, $V(u)$ must be contained in the neighbourhood of the set of parents of the elements in $V(f)$. We observe that for any element $b$ of the postulated basis, the neighbourhood of the parents of $V(b)$ is empty for $b=\cN^0(p)$, $p \in \cP^c \cup \cQ^c$, and atomic in the other cases. If $b$ and $b'$ are distinct elements of the postulated basis, then either both are of the form $(\alpha, \alpha)$, $\alpha \in Q_1$, or, the neighbourhoods of $V(b)$ and $V(b')$ are disjoint as are the neighbourhoods of the sets of their parents. In the former case, $\{(x,x) \mid x \in Q_0\}$ is the neighbourhood of the parents of $V(b)$ and $V(b')$. Since $u$ is atomic and hence has an atomic neighbourhood,  we conclude that either $f$ is atomic or a linear combination of elements of the form $(\alpha, \alpha)$, $\alpha \in Q_1$. Because no atom of $f$ is a coboundary (being a multiple of a basis element and by the first part of this proof), we must be in the second case and the claim follows from \Cref{lem: linear dependencies between traces} and the subsequent discussion.
\end{proof}

\begin{proof}[{Proof of \Cref{Introthm: Hochschild cohomology}}]This follows from \Cref{thm: appendix Hochschild basis} by identifying a complete path or anti-path $p$ of period $m$ with the path $\gamma\colon S^1 \rightarrow \partial \Sigma_A$ which wraps $m$ times around the boundary component $B_p$ following its orientation. Likewise, we identify classes $(p, \para{p})$ and $s(s(v), v)$ ($v$ closed maximal path) with the elements $[B_{p^{-1} \para{p}}]$ and $[B_{vs(v)^{-1}}]$. Finally, any linear combination of classes of the form $(\alpha, \alpha)$, $\alpha \in Q_1$ is identified with an element of $\HH^1(\Sigma_A, \Bbbk)$ via the inverse of the restriction of the embedding $\HH^1(\Sigma_A, \Bbbk) \hookrightarrow \HHH^{1,0}(A,A)$ to its image, cf.~\Cref{lem: linear dependencies between traces}. This identifies the edges in the complement of a spanning graph of $Q$ with a basis of $\HH^1(\Sigma_A, \Bbbk)$. 
\end{proof}
\section{The cup product and the Gerstenhaber bracket}
\noindent Throughout the section, we fix a graded gentle quiver $(Q,I)$ and its associated graded gentle algebra $A$. We first describe the cup product on $\HHH^{\bullet}(A,A)$ and afterwards the Gerstenhaber bracket.

\subsection{Computation of the cup product}

\begin{lem}\label{lem: cup product complete paths and anti-paths}  Suppose $u$ is a primitive complete anti-path or complete path. Then for all $m, n \geq 1$, $$\cN^0(u^m) \cup \cN^0(u^n)=\pm\cN^0(u^{m+n})$$
	$$\cN^0(u^m) \cup \cN^1(u^n)=\pm\cN^1(u^{m+n}).$$
Moreover, if $\alpha \in Q_1$ appears as an arrow in $u$, then $(\alpha, \alpha) \cup \cN^0(u^m)=\pm \cN^1(u^m)$.	
\end{lem}
\begin{proof}We only treat the case of complete anti-paths, the case of complete paths following by analogous arguments. We find that $(u^m, s(u)) \cup (u^n, s(u))=\pm(u^{m+n}, s(u))$ on the chain level. For $p=u^m$, $q=u^n$ and any integers $i,j \in \mathbb{Z}$ such that $i-j$ is not divisible by the period of $u$, we have that $(\rot^i(p),s(\rot^i(p))) \cup (\rot^j(q),s(\rot^j(q)))=0$ on the chain level unless $s(\rot^i(p))=s(\rot^j(q))$ in which case however, the resulting path $\rot^i(p)\rot^j(q)$ is not an anti-path which implies that $(\rot^i(p), s(\rot^i(p)) \cup (\rot^j(q), s(\rot^j(q))=0$ in $C_{\cP}(A)$. This proves the first claim and together with $(p, s(p)) \cup (q\alpha_1(q), \alpha_1)=(pq, \alpha_1(q))$, proves the second assertion. For the final claim, we may assume by rotating $u$ (and hence changing $\cN^0(u^m))$ up to sign) that $u$ starts with the arrow $\alpha$. One then finds,
\begin{equation}\label{eq: cup product torus}
(\alpha, \alpha) \cup \cN^0(u^m)= (-1)^{|u^m| \cdot |\alpha|}(\alpha u^m, \alpha)
\end{equation}
\end{proof}
\noindent From formula \eqref{eq: cup product torus} and its analogue for complete paths, we can conclude the following.
\begin{cor}\label{cor: cup product relation}
Let $p=\alpha_1 \cdots \alpha_l$ be a primitive complete path or anti-path and let $m \geq 1$ be such that $\cN^0(p^m)$ constitutes a non-trivial class in $\HHH^{\bullet}(A,A)$. Then $(\alpha_l, \alpha_l) \cup \cN^0(p^m)=(\alpha_1, \alpha_1) \cup \cN^0(\rot(p^m))$.
\end{cor}

\begin{lem}\label{lem: cup product orthogonal}
Let $b$ and $b'$ be elements of the basis from \Cref{thm: appendix Hochschild basis} of Hochschild cohomology, different from the unit. Then, unless $(b,b')$ is one of the cases from \Cref{lem: cup product complete paths and anti-paths}, one has $b \cup b'=0$.
\end{lem}
\begin{proof}
Suppose $b=\cN^i(u^m)$ and $b'=\cN^j(v^n)$, where $i, j \in \{0,1\}$ where $u=\alpha_1 \cdots \alpha_r$ and $v=\beta_1 \cdots \beta_l$  are both complete primitive anti-paths or both complete primitive paths but $u$ and $v$ belong to distinct rotation classes. Then $\rot^r(u^m) \rot^s(v^n)$ is never an anti-path if $u$ and $v$ are anti-paths and never paths of $A$ if $u$ and $v$ are paths of $A$. Otherwise it would be necessarily complete which would imply that $u$ and $v$ belong to the same rotation class. Thus, $b \cup b'=0$ on the chain level. Likewise, if $p$ is a maximal anti-path, then for any non-trivial anti-path $p'$, neither $pp'$ nor $p' p$ can be an anti-path due to maximality. Likewise, $\para{p}$ is a maximal path of $A$ in this case and so that $\para{p} u \in I$ whenever $u$ is a complete path of $A$. It follows that the chain level cup product of $b$ and $(p, \para{p})$ is zero if $u$ is a complete path or a complete anti-path. 

Next, if $u$ is a complete anti-path and $v$ is a complete path instead so that $(u^m, s(u^m)) \cup (s(v), s(v^n))=\pm (u^m, v^n)$ is non-zero, then $s(u)=s(v)$ and moreover, we necessarily have $\alpha_1=\beta_1$ and $\alpha_r=\beta_s$, from which we conclude that $(u^m, v^n)$ is a coboundary.    
If $p$ and $p'$ are two maximal anti-paths, then $pp' \in I$ and hence $(p, \para{p}) \cup (p', \para{p'})=0$ on the chain level. If $b=(\alpha, \alpha)$ but $\alpha$ does not appear in a complete path or anti-path $u$ and $b'$ is of the form $\cN^s(u)$, it follows $b \cup b'=0$.
\end{proof}

\noindent The following is a generalisation of \cite[Theorem 5.8]{ChaparroSchrollSolotarSuarezAlvarez} to the graded case.
\begin{thm}\label{thm: cup product}
Let $A$ be a graded gentle algebra. Let $\cH$ denote the free graded commutative algebra generated the following elements:
\begin{itemize}
	\item edges $x_e \in Q_1 \setminus T_1$ with $|x_e|=1$, where $T$ is a fixed spanning tree of $Q$;
	\item $x_B$, where $B$ is a fully marked or unmarked boundary component of $\Sigma_A$ and 
	\begin{displaymath}
	|x_B| = \begin{cases}\omega(B), & \text{if $\omega(B) \cdot \operatorname{char}\Bbbk$ is even;} \\ 2\omega(B), & \text{otherwise.} \end{cases}
	\end{displaymath}
	\item $y_C$, where $C$ is a boundary component of $\Sigma_A$ with a single stop and $|y_C|=\omega(C)$.
\end{itemize}
\noindent Then, there is an isomorphism of graded algebras
\begin{displaymath}
\HHH^{\bullet}(A,A) \cong \cH / \mathcal{J}, 
\end{displaymath}
\noindent where $\mathcal{J}$ is the graded ideal generated by all expressions of the following form:
\begin{itemize}
	\item $x_B \cdot x_{B'}$, $B \neq B'$;
	\item $x_B \cdot y_C$; 
	\item $y_C^2$;
	\item $x_e \cdot x_B- x_{f} \cdot x_B$, whenever $e$ and $f$ both appear in the complete path or anti-path corresponding to $B$. 
\end{itemize}
\end{thm}
\begin{proof}Let $u$ be a primitive path or anti-path. Then one identifies $x_{B_u}$ with $\cN^0(u)$ if $\omega(u) \cdot \operatorname{char} \Bbbk$ is even and with $\cN^0(u^2)$ otherwise. Likewise $y_C$ is identified with the class $(p, \para{p})$ or $(s(p), p)$ of the corresponding maximal anti-path or closed maximal path $p$ associated with $C$. Finally, $(\alpha, \alpha)$, $\alpha \in Q_1 \setminus T_1$ is identified with $x_{\alpha}$. Due to the degree conditions on $\omega(u)$ in relation with the characteristic of the field, \Cref{lem: cup product complete paths and anti-paths}, \Cref{cor: cup product relation} and \Cref{lem: cup product orthogonal} imply that there exists a surjective morphism $\cH/\mathcal{J} \rightarrow \HHH^{\bullet}(A,A)$. Then one proceeds as in \cite[Theorem 5.8]{ChaparroSchrollSolotarSuarezAlvarez} to see that this morphism is in fact an isomorphism.
\end{proof}

\subsection{Computation of the Gerstenhaber bracket} 
We apply the formulas \eqref{eq: composition product 1}-\eqref{eq: composition product 3} to compute the Gerstenhaber bracket.  As previously, we endow a product $\prod_{i \geq 1}V_i$ of finite-dimensional graded vector spaces $V_i$ as a complete topological vector space so that the decreasing sequence of subsets $\{\prod_{i \geq j} V_i\}_{j \geq 1}$ is a basis of neighbourhoods around $0$.
\begin{definition}Let $d \in \mathbb{Z}$ and assume that $d \cdot \operatorname{char} \Bbbk$ is even. Define the (completed) \textbf{$d$-graded Witt algebra}
	\begin{equation}\label{eq: graded Witt algebra}
		\widehat{\cW}^+_d \coloneqq \prod_{i \geq 1} \Bbbk[-i \cdot d] \cdot e_i,
	\end{equation}
	\noindent considered as product in the category of graded vector spaces, as the graded Lie algebra with the unique continuous graded Lie bracket such that $[e_i, e_j]=(i-j)e_{i+j}$. We define $\cW^+_d \subseteq \widehat{\cW}^+_d$ as the Lie subalgebra formed by the direct sum of the elements $e_i$ in the category of graded vector spaces.
\end{definition}
\noindent By definition, $\cW^+_0$ is the ordinary (non-complete) Witt algebra and $\widehat{\cW}^+_d$ is concentrated in degrees $d, 2d, 3d, \dots$.  We will see below that elements of the form $\{\cN^1(u^m)\}_{m \geq 1}$ associated to a complete (anti-)path form a copy of a graded Witt algebra which is analogous to the examples discussed in \cite{RedondoRoman}.
\begin{rem} The fact that the product in \eqref{eq: graded Witt algebra} is one of graded vector space means that $\widehat{\cW}^+_d=\cW^+_d$ whenever $d \neq 0$. That we still make this distinction has a deeper origin in our context and is connected to the fact that proper graded gentle algebras are \textit{derived complete}, cf.~\cite{BoothGoodbodyOpper}. It is a generalisation of the fact that the graded polynomial ring $\Bbbk[x]$ with $|x| \in \mathbb{Z}$ is complete as a graded ring with respect to the $(x)$-adic topology if and only if $|x| \neq 0$, while for $|x|=0$, its graded completion is the usual ring  of power series. Graded completion agrees in this case with the derived completion at the simple module $\Bbbk[x]/(x)$.
\end{rem}
\begin{rem}
	The assumption that $d$ or $\operatorname{char} \Bbbk$ are even is necessary to ensure the graded anti-symmetry of the Lie bracket on $\widehat{\cW}^+_d$.
\end{rem}
\noindent The Lie algebras $\cW_d^+$ and $\hat{\cW}_d^+$ admit semi-direct extensions $\cU_d$ and $\hat{\cU}_d^+$. The underlying graded vector space of the latter is
	\begin{equation}\label{eq: graded Witt algebra}
	\widehat{\cU}_d \coloneqq \prod_{i \geq 1} \Bbbk[-i \cdot d] \cdot e_i \times \prod_{i \geq 1} \Bbbk[-i \cdot d+1] \cdot f_i,
\end{equation}
\noindent where $[e_i, e_j]=(i-j)e_{i+j}$ and $[f_i, f_j]=0$ and $[f_i, e_j]=i f_{i+j}$.  As in the case of Witt algebras, $\cU_d$ is defined as the subspace of $\hat{\cU}_d$ obtained by replacing products by direct sums. We note that $\cU_0$ is isomorphic to $\HHH^{\bullet+1}(\Bbbk[x], \Bbbk[x])$, $|x|=0$, as a graded Lie algebra, while $\hat{\cU}_0$ is isomorphic to $\HHH^{\bullet+1}(\Bbbk[t]/(t^2), \Bbbk[t]/(t^2)) \cong \HHH^{\bullet+1}(\Bbbk\llbracket x \rrbracket, \Bbbk\llbracket x \rrbracket)$, where $|t|=1$.

\begin{lem}\label{lem: Lie bracket complete paths and anti-paths}  Let $u$ be a primitive complete anti-path or complete path. Then for all $m, n \geq 1$,
		\begin{displaymath}
			\big[ \cN^1(u^m), \cN^1(u^n)\big] = \begin{cases}(m-n) \cdot \cN^1(u^{m+n}), & \text{if $\omega(u^m)$ and $\omega(u^n)$ are even;} \\ m \cdot \cN^1(u^{m+n}), & \text{if $\omega(u^m)$ is odd and $\omega(u^n)$ is even;} \\ 0, & \text{if $\omega(u^m)$ and $\omega(u^n)$ are odd;} \end{cases}
		\end{displaymath}
		\begin{displaymath}
				[\cN^0(u^m), \cN^1(u^n)] = {\begin{cases}m \cdot  (u^{m+n}, s(u)) & \text{if $\omega(u^n)$ is even;} \\ 0 & \text{if $\omega(u^n)$ is odd and $\omega(u^m)$ is even} \\ (-1)^{m l(u)} \cdot (u^{m+n}, s(u))  &  \text{otherwise.}  \end{cases}}
		\end{displaymath}
	\noindent 	Moreover, if $h \in \HH^1(A,A)$ corresponds to the image of an element $f \in \HH^1(\Sigma_A, \Bbbk)$ and $B_u$ denotes the unmarked or fully marked boundary component of $\Sigma_A$ corresponding to $u$, then $$[h, \cN^1(u^m)]=f(B_u) \cdot \cN^1(u^m).$$ In particular, if $h=(\alpha, \alpha)$ for some $\alpha \in Q_1$, then $f(B_u)$ equals the number of appearances of $\alpha$ in $u^m$. Likewise, if $p$ is a maximal anti-path such that $\para{p}\neq \emptyset$, then $[h, (p, \para{p})]= f(B_p) (p, \para{p})$ and if $q$ is a closed maximal path of $A$, then $[h, (s(q), q)]=f(B_q) (s(q), q)$.
\end{lem}
\begin{proof}We treat the case of complete anti-paths. The case of complete paths is similar. Let $p=\alpha_1 \cdots \alpha_l$ be a primitive complete anti-path of length $l$, $m \geq 0$ and $n \geq 1$.  Then applying our formulas and using $|p^n|=|p| n$ and $|p|^2\equiv |p| \mod 2$, we find that
	\begin{displaymath}
	\begin{aligned}
(\alpha_l p^m, \alpha_l) \circ (\alpha_l p^n, \alpha_l) & = (-1)^{(l(\alpha_lp^m)-1) (l(\alpha_lp^n)+|p^n|-1)} \sum_{j=0}^{m} (-1)^{j (-|p|^{n+
	1}+l(\alpha_lp^n)-1)} (\alpha_l p^{m+n}, \alpha_l) \\ & = {(-1)^{m\cdot  l(p) \cdot \omega(p^n)} \sum_{j=0}^{m} (-1)^{j \omega(p^n)} (\alpha_l p^{m+n}, \alpha_l)} \\ & = {\begin{cases}(m+1) (\alpha_l p^{m+n}, \alpha_l) & \text{if $\omega(p^n)$ is even}; \\ (\alpha_lp^{m+n}, \alpha_l) & \text{if $\omega(p^n)$ is odd and $m$ is even;} \\ 0, & \text{otherwise.} \end{cases}}
\end{aligned} 
	\end{displaymath}
\noindent Note that if $\omega(p^n)=n \omega(p)$ is odd, then so are $n$ and $\omega(p)$ and hence $m$ is even if and only if $\omega(p^m)$ is even. Assuming $m \geq 1$, we conclude that
\begin{displaymath}
\big[ \cN^1(p^m), \cN^1(p^n)\big] = \begin{cases}(m-n) \cdot \cN^1(p^{m+n}), & \text{if $\omega(p^m)$ and $\omega(p^n)$ are even;} \\ m \cdot \cN^1(p^{m+n}), & \text{if $\omega(p^m)$ is odd and $\omega(p^n)$ is even;} \\ 0, & \text{if $\omega(p^m)$ and $\omega(p^n)$ are odd;} \end{cases}
\end{displaymath}
\noindent The value of the bracket in the case when $\omega(p^m)$ is even and $\omega(p^n)$ is odd follows from anti-symmetry. The special case $m=0$, shows \begin{displaymath}
[(\alpha_l, \alpha_l), \cN^1(p^n)]=-(-1)^{0 \cdot \omega(p^n)}[\cN^1(p^n), (\alpha_l, \alpha_l)]=n\cN^{1}(p^n)=f_{\alpha_l}(B_p)\cdot \cN^1(p^n),
\end{displaymath}
\noindent where $f_{\alpha_l} \in \HH^1(\Sigma_A, \Bbbk)$ is a preimage of $(\alpha_l, \alpha_l) \in \HHH^1(A,A)$ under the map from \Cref{lem: linear dependencies between traces}. On the other hand, if $\beta$ is an arrow which does not appear in $p$, then $[(\beta, \beta), \cN^1(p^n)]=0=f_{\beta}(B_p)$, so the second assertion follows by linearity. For the second assertion, we note that $(\alpha_lp^n, \alpha_l) \circ_i (p^m, s(p))=0$ for all $1 \leq i \leq l(\alpha_lp^n)$ and similar as before we have
\begin{displaymath}
	\begin{aligned}
[\cN^0(p^m), \cN^1(p^n)] & =(p^m, s(p)) \circ (\alpha_l p^n, \alpha_l) \\ & = (-1)^{(l(p^m)-1)(l(\alpha_l p^n) + |p^n|-1)} \sum_{i=1}^{m} (-1)^{i (l(\alpha_lp^n)-1) + i |p| |p^n|}  (p^{m+n}, s(p)) \\ &  = (-1)^{(m l(p)-1)\omega(p^n)} \sum_{i=1}^m (-1)^{i \omega(p^n)} (p^{m+n}, s(p)) \\ & = {\begin{cases}m \cdot  (p^{m+n}, s(p)) & \text{if $\omega(p^n)$ is even;} \\ 0 & \text{if $\omega(p^n)$ is odd and $\omega(p^m)$ is even} \\ (-1)^{m l(p)} \cdot (p^{m+n}, s(p))  &  \text{if $\omega(p^n)$ and $\omega(p^m)$ are odd.} \end{cases}}
\end{aligned}
\end{displaymath}
\noindent Next, if $p$ is a maximal anti-path such that $\para{p}\neq \emptyset$, then $(p, \para{p}) \circ (\alpha, \alpha)=(p, \para{p})$ if $\alpha$ appears in $p$ and zero otherwise. Likewise, maximality of $\para{p}$ implies that $(\alpha, \alpha) \circ (p, \para{p})=(p, \para{p})$ if $\alpha$ appears in $\para{p}$. Taking signs into account, we obtain again that $[(\alpha, \alpha), (p, \para{p})]=f_{\alpha}(B_p) \cdot (p, \para{p})$ and the claim follows by linearity. We omit the last and similar case $[(\alpha, \alpha), (s(q), q)]$, where $q$ is a closed maximal path of $A$.
\end{proof}
\begin{rem}
The dependency on the parity of $l(p)$ in the third case of $[\cN^0(u^m), \cN^1(u^n)]$ actually disappears in those cases where $\cN^0(u^m)$ and $\cN^1(u^n)$ are non-trivial Hochschild classes because this case can only appear when $\operatorname{char} \Bbbk$ is even and hence signs do no longer matter. 
\end{rem}

\begin{cor}Let $u$ be a primitive complete anti-path (resp.~path) of $u$. If $\omega(u)$ or $\operatorname{char} \Bbbk$ are even, then the Hochschild classes $\cN^0(u^m), \cN^1(u^m)$, $m \geq 1$ span a graded Lie subalgebra of $\HHH^{\bullet+1}(A,A)$ which is isomorphic to $\hat{\cU}_{\omega(u)}$ (resp.~$\cU_{\omega(u)}$). 
\end{cor}

\begin{lem}Let $b$ and $b'$ be elements of the basis from \Cref{thm: appendix Hochschild basis}. Then, unless the pair $(b,b')$ appears as one of the cases from \Cref{lem: Lie bracket complete paths and anti-paths} (up to switching the roles of $b$ and $b'$), one has $[b,b']=0$.
\end{lem}
\begin{proof}
Let $p, p'$ be primitive complete anti-paths, $q,q'$ primitive complete paths, $u, u'$ distinct maximal anti-paths such that $\para{u}\neq \emptyset \neq \para{u}'$ and $v, v'$ be distinct closed maximal paths of $A$. Suppose $b, b'$ are any two distinct elements from the set 
\begin{equation}\label{eq: orthogonality set}
\{\cN^0(p),\cN^0(p), \cN^0(q), \cN^0(q'), (u, \para{u}),(u', \para{u}'), (s(v), v), (s(v'), v')\}.
\end{equation}
 We recall our standing assumption that $(Q,I)$ is not a loop and not the Kronecker quiver. Then a calculation shows that for $(\alpha, \beta) \in C_{\cP}(A)$, $z \coloneqq (u, \para{u}) \circ (\alpha, \beta)$ is $0$, an atomic coboundary of rank $1$ (whose coboundary property can be verified by \Cref{lem: classification atomic cocycle type A rank 1}), or we have either $z=(\alpha, \para{u})$, $l(u)=1$ and $\beta=u$, or, $z=(u, \para{u})$, $l(u)> 1$ and $\alpha=\beta$ is an arrow which appears in $u$. Hence $(u, \para{u}) \circ w$ is a coboundary for any $w \in \{(u', \para{u}'), \cN^0(p), \cN^0(q), (s(v), v)\}$. Similarly, another calculation shows that $\cN^0(p) \circ b'$ is always zero. This implies that all elements in the set $\{\cN^0(p),\cN^0(p'), (u, \para{u}),(u', \para{u}')\}$ are pairwise orthogonal with respect to $[-,-]$. Since $(\alpha, \beta) \circ (\gamma, \delta)=0$ if $l(\alpha)=0$, we now can also conclude that $\cN^0(q)$ and $(s(v), v)$ are orthogonal to all other elements in \eqref{eq: orthogonality set}. Finally, the compatibility of $\cup$ and $[-,-]$, the fact that $\cN^1(p)$ and $\cN^1(q)$ agree up to sign with the products $(\alpha, \alpha) \cup \cN^0(p)$ and $(\alpha, \alpha) \cup \cN^0(q)$ implies orthogonality in the remaining cases which involve $\cN^1(p)$ and $\cN^1(q)$.  
\end{proof}

\section{Applications to  intrinsically formal gentle algebras}

\noindent We recall that a graded algebra $B$ is \textit{intrinsically formal} if every $A_\infty$-algebra $C$ so that $\HH^{\bullet}(C) \cong B$ as graded algebras, $C$ is formal, that is, quasi-isomorphic to its cohomology. The following theorem characterises intrinsically formal graded gentle algebras under a mild assumption. 
\begin{thm}\label{thm: intrinsic formality graded gentle algebras}Let $A$ be a graded gentle algebra. Suppose $\Sigma_A$ contains no boundary components which has a single stop and winding number $2$. Then $A$ is intrinsically formal if and only if $\Sigma_A$ has no unmarked boundary component with winding number $2$.
\end{thm}
\begin{proof}We assume that the reader is familiar with the terminology from \cite{HaidenKatzarkovKontsevich}. First, lets prove that $A$ is not intrinsically formal if $\Sigma_A$ contains an unmarked component of winding number $2$. The assumption one the absence of certain components with a single marked interval is only needed in the proof other direction. If $B$ is an unmarked component $B$ with winding number $2$, it means that $A$ has a primitive complete anti-path $u=\alpha_1 \cdots \alpha_n$ such that $\omega(u)=2$ and by gluing a disk to $B$ we obtain a graded marked surface $\Sigma' \supseteq \Sigma_A$. Here the assumption on the winding number of $B$ is used to guarantee that the line field on $\Sigma_A$ extends to a line field on $\Sigma'$. The algebra $A$ corresponds to an arc system $\cA$ on $\Sigma'$ (and even on $\Sigma_A$), where the presernce ofthe anti-path $u$ implies the existence of a disk sequence. By \cite[Section 3.3]{HaidenKatzarkovKontsevich}, $A$ can therefore be equipped with a minimal $A_\infty$-structure with non-trivial higher $A_\infty$-multiplications which extends the multiplication on $A$. Denoting this $A_\infty$-algebra by $A'$, it follows from the discussion preceding \cite[Lemma 3.2]{HaidenKatzarkovKontsevich} that there is a non-trivial relation in $\cK_0(A')$ between elements the classes $[x]$, $x \in Q_0$, associated to the vertices of $A'$ (=vertices of $A$). If $A$ is homologically smooth, one can argue as in \cite[Theorem 5.1]{HaidenKatzarkovKontsevich} that $\cK_0(A)$ is freely generated by the classes $[x]$ as an abelian group and hence $\cK_0(A) \not \cong \cK_0(A')$. It follows that $A'$ cannot be formal. Every graded gentle algebra is an idempotent subalgebra of a homologically smooth graded gentle algebra simply by turning all unmarked components of $\Sigma_A$ into fully marked ones and adding additional arcs to the arc system $\cA$ to produce a full formal arc system of of this new surface. Because subgroups of free abelian subgroups are still free, $\cK_0(A)$ is free of rank $|\cA|$ and we can produce a contradiction in the same way. 

For the converse, we use Kadeishvili's criterion \cite{Kadeishvili}. It says that a graded algebra $B$ such that $\HHH^{n,2-n}(B,B)=0$ for all $n \geq 3$ is intrinsically formal. First we note that all contributions of fully marked components and closed maximal paths of $A$ lie in $\HHH^{s, \bullet}(A,A)$, $s \in \{0,1\}$ and hence are irrelevant for us. The same applies to $\HH^1(\Sigma_A, \Bbbk)$ whose contributions lie in $\HHH^{1, 0}(A,A)$. Due to our assumptions, we therefore conclude that $\HHH^{n,2-n}(A,A)=0$ for all $n \geq 2$ and the claim follows from Kadeishvili's criterion.
\end{proof}
\noindent A peculiar consequence of the previous theorem is that intrinsic formality is a Morita invariant for graded gentle algebras satisfying the assumptions of \Cref{thm: intrinsic formality graded gentle algebras} (which is invariant under derived equivalence), meaning that such a graded gentle algebra $A$ is intrinsically formal if and only this is the case for all derived equivalent graded gentle algebras. This is far from true for general graded algebras.
\begin{rem}\label{rem: intrinsic formality}
Computations suggest that \Cref{thm: intrinsic formality graded gentle algebras} remains true when removing the assumption on $\Sigma_A$ from \Cref{thm: intrinsic formality graded gentle algebras} about boundary components with a single stop. The corresponding Hochschild cocycle $f=(\alpha, \para{\alpha})$ is a Maurer-Cartan element in the Hochschild complex with respect to dg Lie algebra structure. Thus $(A, f)$ defines a non-trivial minimal $A_\infty$-structure on $A$  which appears to be formal via a \textit{formal diffeomorphism} in the sense of \cite{SeidelBook}. This seems to extend to linear combinations of elements of the same form. In the absence of unmarked components of winding number $2$, $\HHH^{n, 2-n}(A,A)$ is generated by such elements for all $n \geq 2$. We expect that an adaptation of the proof of Kadeishvili's formality criterion could be used to show that $A$ is still intrinsically formal.
\end{rem}

\bibliography{Bibliography}{}
\bibliographystyle{alpha}
%\printbibliography
\end{document}